\documentclass[11pt,uplatex]{article}

\usepackage{amsmath,amssymb,amsthm,bm,mathrsfs,braket,color,sasaki}
\usepackage{graphicx}

\makeatletter
\@addtoreset{equation}{section}
\makeatother

\theoremstyle{plain}
\newtheorem{theorem}{Theorem}[section]
\newtheorem{proposition}[theorem]{Proposition}
\newtheorem{lemma}[theorem]{Lemma}

\newtheorem{corollary}[theorem]{Corollary}
\newtheorem{remark}[theorem]{Remark}
\newtheorem{assumption}[theorem]{Assumption}

\newtheorem{example}[theorem]{Example}

\newcommand{\Spec}{\mathop{\mathrm{Spec}}\nolimits}

\newcommand{\R}{\mathbb R}

\DeclareMathOperator{\Dom}{Dom}

\begin{document}


\setlength{\baselineskip}{15pt}

\date{}

\title
{\Large \sc Absence of Embedded Eigenvalues for Non-Local Schr\"odinger Operators}
\vspace{1.5cm}
\author{
\small Atsuhide Ishida \\[0.1cm]
{\small\it  Katsushika Division, Institute of Arts and Sciences} \\[-0.4ex]
{\small \it  Tokyo University of Science, Tokyo, 125-8585, Japan}    \\[-0.4ex]
{\small  {\tt aishida@rs.tus.ac.jp}   }\\[0.5cm]
\small J\'ozsef L\H{o}rinczi \\[0.1cm]
{\small\it Alfr\'ed R\'enyi Institute of Mathematics }    \\[-0.4ex]
{\small\it Re\'altanoda utca 13-15, 1053 Budapest, Hungary}      \\[-0.4ex]
{\small  {\tt lorinczi@renyi.hu}   }\\[0.5cm]
\small Itaru Sasaki \\[0.1cm]
{\it \small Department of Mathematics, Shinshu University} \\[-0.4ex]
{\it \small Matsumoto, Asahi 3-1-1, Japan} \\[-0.4ex]
{\small {\tt  isasaki@shinshu-u.ac.jp}} \\[-0.4ex]}

\maketitle

\bigskip

\begin{abstract}
\noindent
We consider non-local Schr\"odinger operators with kinetic terms given by several different types of functions
of the Laplacian and potentials decaying to zero at infinity, and derive conditions ruling embedded eigenvalues
out. Our goal in this paper is to advance techniques based on virial theorems, Mourre estimates, and an extended
version of the Birman-Schwinger principle, previously developed for classical Schr\"odinger operators but thus
far not used for non-local operators. We also present a number of specific cases by choosing particular classes
of kinetic and potential terms, and discuss existence/non-existence of at-edge eigenvalues in a
basic model case in function of the coupling parameter.


\medskip
\noindent
2010 MS Classification: primary 47A75, 47G30; secondary 34L40, 47A40, 81Q10
\end{abstract}

\medskip

\makeatletter
\renewcommand\@dotsep{10000}
\makeatother

\section{Introduction}

Classical Schr\"odinger operators $H = -\frac{1}{2}\Delta + V$, featuring the Laplacian $\Delta$ and a multiplication operator
$V$ called potential continues to be a much researched subject in mathematical physics. When the kinetic term is changed for a
suitable pseudo-differential operator, the appropriately defined sum gives rise to a non-local Schr\"odinger operator. Operators
of this type attract increasing attention for their interest in functional analysis and stochastic processes, and their relevance
in various models in mathematical physics.

In \cite{HIL12a} we introduced a class of non-local Schr\"odinger operators $H = H_0 + V$ with kinetic term $H_0=\Phi(-\Delta)$
given by a so-called Bernstein function $\Phi$ of the Laplacian (for details see Remark \ref{counter} below), and with Kato-class
or more singular potentials.
This has the special interest that Bernstein
functions describe the L\'evy-Khintchine exponents of subordinators, and thus an analysis of the semigroup $\{ e^{-tH}: t\geq 0\}$
via Feynman-Kac techniques becomes possible by using subordinate Brownian motion. Specific choices include key operators in
mathematical physics, which will be our main examples in this paper, such as fractional (i.e., massless relativistic) Schr\"odinger
operators with $H_0 = (-\Delta)^{\alpha/2}$, $0 < \alpha < 2$, relativistic Schr\"odinger operators with $H_0 = (-\Delta +
m^{2/\alpha})^{\alpha/2}- m$, where $m \geq 0$ is the rest mass of the particle, and many others involving further applications
such as anomalous transport. The relativistic (square root Klein-Gordon) operator obtained for $\alpha = 1$ and $m > 0$ is an
early motivation for the study of spectral properties of such non-local Schr\"odinger operators \cite{W74,H77,CMS,LS10}. For further
interest in these operators we refer to \cite{KM,BL19a,BL19b,DLN}.

As it is well-known already from classical Schr\"odinger operators, for potentials decaying to zero at infinity the spectrum of
$H$ can be very intricate, involving a rich phenomenology. Going back to a paper by Wigner and von Neumann published in 1929, it
is also known that eigenvalues embedded in the absolutely continuous spectrum may exist. This is a long range effect due to a
combination of slow decay at a rate $O(1/|x|)$ and oscillations of the potential. The study of such potentials has intensified
over the decades, leading even to experimental realisations \cite{C}; for summaries see \cite{RS3,EK} and the references therein.

The related problem of non-existence of embedded eigenvalues is of an equal interest, triggered by a pioneering result by Kato
\cite{K59} showing that if $V(x) = o(1/|x|)$, then no embedded eigenvalues exist above the edge of the continuous spectrum. Also,
it is a fundamental question if for a given potential a bound state exists at the spectral edge, i.e., whether zero may be an
eigenvalue. This is a borderline case giving rise to a phenomenology of its own. For instance, the occurrence of zero-resonances
is known to be a key condition for the Efimov effect, leading to the emergence of a countable set of bound states in an interacting
three-particle quantum system, with eigenvalues accumulating at zero. On the other hand, in case no zero-eigenvalues occur, specific
dispersive estimates on the time evolutions of projections to the continuous spectrum under the unitary Schr\"odinger semigroup can
be obtained yielding Strichartz bounds \cite{Sch07}, or in another context, time operators exist \cite{A09,AH}.

For non-local Schr\"odinger operators the occurrence of zero or strictly positive eigenvalues just begins to be studied. In
\cite{LS17} we have constructed Neumann-Wigner type potentials for the operator $(-\Delta + m^2)^{1/2} - m$, for sufficiently
large $m > 0$, giving rise to embedded eigenvalues equal to $\sqrt{1+m^2}-m$. Also, we have shown that they converge to classical
cases of Neumann-Wigner type potentials in the non-relativistic limit. We have also obtained two families of fractional
Schr\"odinger operators in one-dimension for the case $m=0$ for which zero-eigenvalues occur. In \cite{JL18} we have gone beyond
this and constructed potentials on $\R^d$, $d\geq 1$, for $(-\Delta)^{\alpha/2}$ with arbitrary order $0 < \alpha < 2$; for further
details see Section 5 below.

In a further study \cite{AL} we have undertaken a systematic description of potentials generating eigenvalues at zero for
non-local Schr\"odinger operators with kinetic term $H_0=\Phi(-\Delta)$, given by Bernstein functions $\Phi$. Our results
reveal a specific interplay between the sign at infinity and the decay at infinity of such potentials, and a qualitative
difference in the mechanisms between non-local operators with heavy-tailed and exponentially light-tailed jump measures,
respectively (e.g., massless vs massive relativistic Laplacians).
Furthermore, by using a  Feynman-Kac formula-based analysis, in \cite{KL18} we derived sharp estimates on the spatial decay
of eigenfunctions at zero-eigenvalue or zero-resonance for fractional and comparable Schr\"odinger operators, and identified
possible scenarios of decay, which are significantly different from cases of bound states at negative eigenvalues \cite{KL17}.
For some other developments in the direction of existence of strictly positive embedded eigenvalues we refer to
\cite{M06,RUU16,C17}, to \cite{U06,RU12} for scattering theory of relativistic operators with $\alpha=1$, and  \cite{IW}
for a scattering theory for a more general class.

In the present paper our focus is on the non-existence aspect of embedded eigenvalues by using purely operator theory methods.
Apart from implications of absence of at-edge eigenvalues under specific conditions shown in \cite[Sect. 5]{AL}, which is an
independent approach from other techniques, some results in this direction have been obtained for fractional Schr\"odinger
operators with $\alpha = 1$ and dimension $d=3$ in \cite{RU12}, and more in \cite{FF14,S14,S15,R15,R17} by using unique
continuation arguments. It is natural to ask how do the criteria of absence of embedded eigenvalues observed for classical
Schr\"odinger operators modify for non-local cases, and as a first step in this direction, our goal is to explore how far
stock methods such as virial theorems, Mourre estimates, and an extension of the Birman-Schwinger principle to include the
edge of the continuous spectrum can be developed to lead to results on non-existence of embedded eigenvalues. Fitting with
these techniques, in the remainder of the paper we will consider three different frameworks of non-local operators
(some going beyond Bernstein functions of the Laplacian, and derive conditions ruling out either positive or
zero eigenvalues.

Our first set of results is based on a counterpart of the virial theorem (Section 2). We consider a class of operators of
the form $H= \ome(-i\nabla) + V$, with conditions on $\omega$ and $V$ fixed in Assumption \ref{viri_edited_by_atsuhide}.
This class includes massless and massive relativistic Schr\"odinger operators, and other cases in which the kinetic term is
given by a class of Bernstein functions, but it does not cover cases when $\omega$ is slowly varying (see Example \ref{EX1}
and Remark \ref{counter}). Our main result in this section is a virial-type theorem showing that no eigenvalues beyond zero
exist (Theorem \ref{thm2}). In case when $\ome(-i\nabla)=(-\Delta)^{\alpha/2}$, this translates to non-existence of positive
eigenvalues for a negative potential like $V(x) = - C (1+|x|^2)^{-\beta/2}$, $C>0$, whenever $0<\beta \leq \alpha$ (Example
\ref{EX3_edited_by_atsuhide}).

A second way we pursue is to obtain a Mourre estimate for non-local Schr\"odinger operators of the form $H =
\Psi(-\frac{1}{2}\Delta) + V$ on $L^2(\R^d)$ (Section 3). First we investigate the pure point spectrum, under the
conditions of Assumption \ref{ass1_atsuhide} on $\Psi$ and $V$. In Lemma \ref{lem1_atsuhide} we prove that $H$ has the $C^1(A)$
property with respect to the conjugate operator $A=-\frac{i}{2}(x\cdot \nabla + \nabla \cdot x)$, i.e., a more convenient
operator than used in \cite{Mo} in the initial attempts to apply the theory for pseudo-differential symbols.
Then we
prove in Theorem \ref{thm1_atsuhide} that $H$ has a pure point spectrum containing eigenvalues of finite multiplicity, with
only possible accumulation point at zero. Our strategy is based on \cite{AmBoGe}. Under further conditions
(see Assumption \ref{ass2_edited_by_atsuhide}) we then prove that this continues essentially to hold (Theorem \ref{thm2_atsuhide}).
We discuss fundamental models, such as the fractional Laplacian $\Psi(-\frac{1}{2}\Delta)=(-\Delta)^{\alpha/2}$
and the massive relativistic operator $\Psi(-\frac{1}{2}\Delta)=(-\Delta+m^{2/\alpha})^{\alpha/2}-m$ as specific cases, showing
finite multiplicity of eigenvalues above zero in Examples \ref{ex1_atsuhide}-\ref{ex2_atsuhide} and \ref{ex3_atsuhide}. In Theorems
\ref{MG1} and \ref{MG2} we arrive at showing that the operators $H$ in the given class have no embedded eigenvalue above zero.
Applying these theorems, we give some concrete potentials (see also Examples \ref{ex4_atsuhide} and \ref{EX6}-\ref{ex_itaru}).

Thirdly, we apply an extended version of Birman-Schwinger principle \cite{Sim} to non-local Schr\"odinger operators
with Bernstein functions of the Laplacian (Section 4), which differs from the classic version by including zero (the edge of
the essential spectrum).
Under a smallness condition on the potential, we show in Theorem \ref{thm4_atsuhide} that
in the case of the fractional Laplacian for $d\geq3$ with bounded and compactly supported
potentials no eigenvalues occur at zero or above. Corollary \ref{cor1_atsuhide} extends this conclusion to
fully supported potentials under an extra smallness condition on its $L^1$-norm. In the remaining part of the section, we
focus on the massive relativistic operator for $d=3$ and $1<\alpha< 2$ to prove by using its resolvent that there exist no
eigenvalues in $[0,\infty)$.

In the concluding Section 5 we study the existence and absence of embedded eigenvalues for the model case $H = (-\Delta)^
{\alpha/2} - C_\alpha (1+|x|^2)^{-\alpha}$ on $L^2(\R^3)$, in function of the coupling parameter $C_\alpha > 0$. In Theorem
\ref{dicho} we show that (under a restriction on $\alpha$), when $C_\alpha$ is small enough, no eigenvalue at the spectral
edge occurs, while it does for specific larger choices of $C_\alpha$. Also, we prove a property on the location of the pure
point spectrum and the existence of a discrete spectrum in two different regimes of $C_\alpha$.

Finally we note that we will adopt the notation $\Spec(A)$ throughout for the spectrum of operator $A$, with
the subscripts ``ess" for the essential,
``p" for the pure point, ``ac" for the absolutely continuous, and ``d" for the discrete
components of the spectrum, and mean the full spectrum when no subscript appears. Also, we write $\Dom (A)$ to denote the
domain of operator $A$, two-slot pointed brackets denote usual scalar product in $L^2(\R^d)$, and for one-slot we also use
the standard shorthand $\langle\cdot\rangle = \sqrt{1+|\cdot|^2}$.

\section{Virial theorem for non-local Schr\"odinger operators}
Let $V$ and $\ome$ be real-valued Borel functions on $\RR^d$. In this section we consider non-local Schr\"odinger
operators of the form
\begin{align*}
    H = \ome(-i\nabla) + V,
\end{align*}
acting on $L^2(\RR^d)$. For a function $f: \R^d \to \R$, we define the scaling
\begin{equation}
\label{scal}
f_a(x)=f(ax), \quad x \in \R^d.
\end{equation}
We will use a framework given by the basic restrictions below.
\begin{assumption}
\label{viri_edited_by_atsuhide}
The following conditions hold:
\begin{itemize}
\item[\rm (A.1)]
For almost every $\xi \in \RR^d$, $\ome(\xi) \geq 0$ is non-negative, $\ome(a\xi)$ is non-decreasing in $a\in
(0,1)$, and $\ome(\xi)$ is once differentiable with the bound $\xi\cdot(\nabla_\xi\ome)(\xi)\leq c\ome(\xi)$
for some $c>0$.
\item[\rm (A.2)]
The function
\begin{align*}
   b(a) := \sup_{\xi\in\RR^d} \frac{\ome(a\xi)}{\ome(\xi)}, \qquad a\in\Big[\frac{1}{2},1\Big),
\end{align*}
is well-defined, once differentiable in a neighbourhood of $a=1$, and $\lim_{a\uparrow1}b'(a) > 0$ exists.
\item[\rm (A.3)]
$V$ is relatively bounded with respect to $\ome(-i\nabla)$ with relative bound less than 1.
\item[\rm (A.4)]
There exists a multiplication operator $W$ in $L^2(\RR^d)$ with $\Dom(W)\supset \Dom(\ome(-i\nabla))$,
such that for all $\psi \in \Dom(\ome(-i\nabla))$
\begin{align*}
     \frac{1}{1-a}(V-V_a)\psi \to W\psi \quad \text{as}\quad a \uparrow 1
\end{align*}
in weak convergence sense.
\end{itemize}
\end{assumption}

Note that by (A.3) and the Kato-Rellich theorem $H$ is self-adjoint on $\Dom(\ome(-i\nabla))$. The following is the
main theorem in this section and corresponds to \cite[Th. XIII 59, Th. XIII 60]{RS4} for the case $\ome(-i\nabla)=-\Delta$.

\begin{theorem}
{\label{thm2}}
Let Assumption \ref{viri_edited_by_atsuhide} hold, and define $q =1/\lim_{a\uparrow1}b'(a)$. Suppose that
\begin{align}
\inner{\psi}{(V+qW)\psi}<0, \quad \psi \in \Dom(\ome(-i\nabla))\setminus\{0\}.\label{thm2_ass}
\end{align}
Then $H$ has no eigenvalue in $(0,\infty)$. Moreover, if $V$ is relatively compact with respect to
$\omega(-i\nabla)$, then the essential spectrum of $H$ is $[0,\infty)$ and $H$ has no embedded eigenvalue above zero.
\end{theorem}
\begin{proof}
Let $\lambda$ be an arbitrary eigenvalue of $H$ with corresponding eigenfunction $\psi$, $\|\psi\|_2=1$, and denote
$p=-i\nabla$. By (A.1) we have that $\Dom(\ome(-i\nabla))\subset \Dom(\ome(-ai\nabla))$. By a basic property of
difference operators we get that $\psi_a \in \Dom(\ome(p))$ and
\begin{align*}
  \ome(p)\psi_a = b(\ome(p)\psi)_a + (\ome_a(p) - b\ome(p))\psi)_a.
\end{align*}
holds for every $a \in (0,1)$ and $b\in\RR$, using the scaling \eqref{scal}. Since $\ome(p)\psi = -V(x)\psi +
\lambda \psi$, we have
\begin{align*}
  \ome(p)\psi_a = b ((\lambda-V)\psi)_a + ((\ome_a(p) - b\ome(p))\psi)_a,
\end{align*}
and thus
\begin{align}
   \inner{\psi}{\ome(p)\psi_a}
   =
   b \inner{\psi}{(\lambda -V_a)\psi_a} + \inner{\psi}{((\ome_a(p) - b\ome(p))\psi)_a}
\label{bz}
\end{align}
Since $\ome(p)$ is self-adjoint, \eqref{bz} becomes
\begin{align*}
   \inner{\psi}{(\lambda-V)\psi_a}
   \, = \,
   b \inner{\psi}{(\lambda -V_a)\psi_a} + \inner{\psi}{((\ome_a(p)-b\ome(p))\psi)_a}.
\end{align*}
Hence for every $b\neq 1$ it follows that
\begin{align}
 \lambda \inner{\psi}{\psi_a}
  = \, & \frac{1}{1-b}\inner{\psi}{(V-V_a)\psi_a} + \inner{\psi}{V_a\psi_a}  \notag \\
    \, & + \frac{1}{1-b} \inner{\psi}{((\ome_a(p)-b\ome(p))\psi)_a}.  \label{2.8}
\end{align}
By a denseness argument, we get that $\psi_a \to \psi$ as $a\uparrow 1$ in strong convergence sense. Using (A.4) we get
\begin{align*}
\inner{(V-V_a)\psi}{\varphi}=(1-a)\inner{\left(\frac{V-V_a}{1-a}-W\right)\psi}{\varphi}+(1-a)
\inner{W\psi}{\varphi} \to 0
\end{align*}
as $a\uparrow 1$ for $\varphi\in L^2({\RR^d})$. We therefore have $\sup_{a\in[1/2,1)}\|V_a\psi\|<\infty$
and
\begin{align}
\inner{\psi}{V_a\psi_a}-\inner{\psi}{V\psi}=\inner{(V_a-V)\psi}{\psi}+\inner{V_a\psi}{(\psi_a-\psi)}\to 0
\label{2.11}
\end{align}
as $a\uparrow 1$. Note that (A.1)-(A.2) imply that $b(a)<1$ in a neighbourhood of $a=1$ and
\begin{align*}
1>b(a)\geq\frac{\ome(a\xi)}{\ome({\xi})} \to 1
\end{align*}
holds as $a\uparrow 1$, for almost every $\xi\in\RR^d$. We thus extend $b(a)$ to $a=1$ by $b(1)=
\lim_{a\uparrow1}b(a)=1$. In a neighbourhood of $a=1$ we have
\begin{align}
\frac{1-b(a)}{1-a}=\frac{b(a+1-a)-b(a)}{1-a}=b'(a)+O(1-a)
\quad \mbox{and} \quad
\lim_{a\uparrow1}\frac{1-b(a)}{1-a}=\lim_{a\uparrow1}b'(a)=\frac{1}{q}
\label{lim_b'}
\end{align}
Write now $b = b(a)$ in \eqref{2.8}. By (A.4), \eqref{lim_b'} and strong convergence of $\psi_a\to\psi$ we
have
\begin{align}
   & \lim_{a \uparrow 1} \frac{1}{1-b(a)} \inner{\psi}{(V-V_a)\psi_a}  \notag \\
   & \quad = \lim_{a \uparrow 1} \frac{1-a}{1-b(a)} \frac{1}{1-a} \inner{\psi}{(V-V_a)\psi_a}
   = \inner{\psi}{qW(x)\psi},
\label{2.15}
\end{align}
where we used that
$\sup_{a\in[1/2,1)}\big\|\frac{V-V_a}{1-a}\psi\big\|<\infty$,
resulting from (A.4). By (A.1)-(A.2),
\begin{eqnarray*}
\frac{\ome(\xi)-\ome(a\xi)}{1-a}
&=&
\frac{1}{1-a}\int_a^1\xi\cdot(\nabla_\xi\ome)(\theta\xi)d\theta \notag \\
&\leq&
\frac{c}{1-a}\int_a^1\frac{\ome(\theta\xi)}{\theta}d\theta\leq\frac{c}{(1-a)a}\int_a^1b(\theta)
d\theta\ome(\xi)\leq2c\ome(\xi)
\end{eqnarray*}
and
\begin{align*}
\frac{\ome(\xi)-\ome(a\xi)}{1-a}\to\xi\cdot(\nabla_\xi\ome)(\xi)
\quad\mbox{as}\quad a\uparrow1
\end{align*}
hold for almost every $\xi\in\RR^d$. Since $\sup_{a\in [1/2,1)}|(1-a)/(1-b(a))|<\infty$ by \eqref{lim_b'},
and
\begin{align*}
\frac{\ome(a\xi)-b(a)\ome(\xi)}{1-b(a)}=\ome(\xi)-\frac{1-a}{1-b(a)}\frac{\ome(\xi)-\ome(a\xi)}{1-a},
\end{align*}
we obtain
\begin{align}
\left|\frac{\ome(a\xi)-b(a)\ome(\xi)}{1-b(a)}\right|
\leq
\ome(\xi)+2c\sup_{a\in [1/2,1)}\left|\frac{1-a}{1-b(a)}\right|\ome(\xi).
\label{2.16}
\end{align}
By (A.1) and
\eqref{2.16}, we have
\begin{align*}
\left|\left(\frac{\ome(a\xi)-b(a)\ome(\xi)}{1-b(a)}-\ome(\xi)+q\xi\cdot(\nabla_\xi\ome)(\xi)\right)
\hat{\psi}(\xi)\right|\leq C\ome(\xi)|\hat{\psi}(\xi)|,
\end{align*}
where $\hat{\psi}$ denotes the Fourier transform of $\psi$. Since $\ome\hat{\psi}\in L^2(\RR^d)$, dominated convergence then gives
\begin{align}
\frac{\ome(ap)-b(a)\ome(p)}{1-b(a)}\psi\to(\ome(p)-qp\cdot(\nabla_\xi\ome)(p))\psi\label{2.17}
\end{align}
in strong convergence sense as $a\uparrow1$. Noting that $\inner{f_a}{g_a}=a^{-d}\inner{f}{g}$, we thus obtain
\begin{align}
\frac{1}{1-b(a)} \inner{\psi_a}{((\ome_a(p)-b(a)\ome(p))\psi)_a}
\to\inner{\psi}{\left(\ome(p)-qp\cdot(\nabla_\xi\ome)(p)\right)\psi}\label{2.18}
\end{align}
as $a\uparrow1$. We write
\begin{align}
\lefteqn{\frac{1}{1-b(a)} \inner{\psi}{((\ome_a(p)-b(a)\ome(p))\psi)_a}}\nonumber \\
&=\frac{1}{1-b(a)} \inner{\psi_a}{((\ome_a(p)-b(a)\ome(p))\psi)_a}+\frac{1}{1-b(a)} \inner{\psi-\psi_a}{((\ome_a(p)-b(a)\ome(p))\psi)_a}.\label{2.19}
\end{align}
The second term at the right hand side of \eqref{2.19} is
\begin{align}
\frac{1}{1-b(a)} \left|\inner{\psi-\psi_a}{((\ome_a(p)-b(a)\ome(p))\psi)_a}\right|\leq\|\psi-\psi_a\|a^{-d/2}\left\|\frac{\ome(ap)-b(a)\ome(p)}{1-b(a)}\psi\right\|\to0\label{2.20}
\end{align}
as $a\uparrow1$, where we used $\|f_a\|=a^{-d/2}\|f\|$ and the strong convergence \eqref{2.17}. By \eqref{2.18}-\eqref{2.20}, we have
\begin{align}
-\infty<\lim_{a\uparrow 1} \frac{1}{1-b(a)} \inner{\psi}{((\ome_a(p)-b(a)\ome(p))\psi)_a}
=\inner{\psi}{\left(\ome(p)-qp\cdot(\nabla_\xi\ome)(p)\right)\psi} \leq 0,
\label{2.21}
\end{align}
noting that using (A.2) the operator inequality $\ome(ap)\leq b(a)\ome(p)$ follows. By \eqref{thm2_ass},
\eqref{2.8}-\eqref{2.11}, \eqref{2.15} and \eqref{2.21}, we have then $\lambda\inner{\psi}{\psi} < 0$,
hence $H$ has no non-negative eigenvalues.
\end{proof}

Assumptions (A.1) and (A.2) are conditions on the kinetic term $\ome(p)$ in $H$. Theorem \ref{thm2} can be applied
to various specific choices of non-local Schr\"odinger operators.
\begin{example}{\label{EX1}}
{\rm
Some examples of $\ome(p)$ of immediate interest and the values of the corresponding coupling parameter
$q=1/\lim_{a\uparrow 1}b'(a)$ are as follows:
\begin{enumerate}
\item[(1)]
$\ome(p) = -\frac{1}{2}\Delta$, with $q=\frac{1}{2}$ (classical Schr\"odinger operator)
\item[(2)]
$\ome(p) = (-\Delta)^{\alpha/2}$, $0 < \alpha < 2$, with $q=\frac{1}{\alpha}$ (fractional Schr\"odinger operator)
\item[(3)]
$\ome(p) = (-\Delta + m^{2/\alpha})^{\alpha/2} - m$, $0 < \alpha < 2$, $m > 0$, with $q=\frac{1}{\alpha}$  (relativistic
Schr\"odinger operator)
\item[(4)]
$\ome(p) = -\Delta +c(-\Delta)^{\alpha/2}$, $c > 0$, with $q=\frac{1}{2}$ (jump-diffusion operator).
\end{enumerate}
}
\end{example}

\begin{remark}
\label{counter}
{\rm
For their special interest, we recall that the Bernstein functions mentioned in the Introduction have the canonical
description
\begin{equation}
\Phi(u) = k + bu + \int_{(0,\infty)} (1 - e^{-yu}) \mu(dy), \quad u \geq 0,
\label{bernrep}
\end{equation}
in terms of three parameters, two numbers $k, b\geq 0$, and a Borel measure $\mu$ on $\R \setminus \{0\}$ such that
$\mu((-\infty,0))=0$ and $\int_{\R\setminus\{0\}} (y \wedge 1) \mu(dy) < \infty$. The non-local operator $\Phi(-\Delta)$
can be constructed by using functional calculus, and the case $k=0=b$ corresponds to a purely jump component. The choices
of $\omega$ in Example \ref{EX1} are Bernstein functions as given in \eqref{bernrep}, however, we also note that not every
Bernstein function satisfies the conditions set in Assumption \ref{viri_edited_by_atsuhide}. While it is known that $\Phi(u)
= \log \left(1 + u^{\alpha/2}\right)$ is a Bernstein function, choosing $\ome(p) = \log \left(1 + (-\Delta)^{\alpha/2}\right)$,
$0<\alpha<2$, we see that $b(a) = 1$. Hence $b'(a)= 0$ and condition (A.2) fails to hold.
}
\end{remark}

The above result can be applied to specific cases of interest.
\begin{example}{\label{EX2}}
{\rm
Let $d=3$ and $\ome(p)=\sqrt{-\Delta}$. Consider the homogeneous Coulomb-type potential $V(x) = -C|x|^{-\gamma}$,
with $C>0$ and $\gamma \in (0,1)$. Then $V$ is $\sqrt{-\Delta}$-compact with relative bound less than 1 by
\cite[Prop. p. 170]{RS2} and \cite[Ex. 7 p. 118]{RS4}. Moreover,
\begin{align*}
  \frac{1}{a-1}(V_a(x)-V(x))
     & = \frac{1}{a-1} \frac{-C}{|x|^\gamma} \left(\frac{1}{a^\gamma}-1\right) \\
     & = V(x) a^{-\gamma} \frac{1-a^\gamma}{a-1} \; \stackrel {a \uparrow 1} \longrightarrow \;  -\gamma V(x) =: W(x).
\end{align*}
Then $V+qW= V-\gamma V = (1-\gamma)V <0$. Thus by Theorem \ref{thm2} the operator $H = \ome(p)+V$ has no positive eigenvalue.
}
\end{example}

\begin{example}\label{EX3_edited_by_atsuhide}
{\rm
Let $\ome(p)=(-\Delta)^{\alpha/2}$ with $\alpha>0$, and consider 
\begin{equation*}
V(x)=-\frac{C}{(1+|x|^2)^{\beta/2}}
\end{equation*}
with $C>0$ and $0<\beta\leq\alpha$.
We have $W(x)=x\cdot \nabla V(x)=C\beta|x|^2\langle x\rangle^{-\beta-2}$. Then
$$
V+qW=-C\langle x\rangle^{-\beta-2} (1+(1-\beta/\alpha)|x|^2)<0
$$
and thus $H = \ome(p)+V$ has no positive eigenvalue. We will obtain an improved result
for this case in Example \ref{ex6_atsuhide}.
}
\end{example}

\section{Mourre estimate for non-local Schr\"odinger operators}
In this section we consider the operator
\begin{align}
   H = H_0 + V := \Psi(p^2/2) + V
\end{align}
acting on $L^2(\RR^d)$, where $\Psi : [0,\infty) \to \RR$ is a real measurable function subject to the
conditions below, and $p=-i\nabla$. Our goal here is to develop the general ideas in \cite{Mo}
to the framework given by a class of non-local Schr\"odinger operators defined by the conditions below. As
it will be seen, we will be able to use a conjugate operator which is simpler than those originally used in
the context of pseudo-differential operators, and we derive our results by the virial theorem in the context
of Mourre theory completed with a virial condition, arriving at some verifiable criteria which can be checked
concretely on specific cases.

First we investigate the pure point spectrum of $H$. In this case $V$ may have singularities and satisfy
rather general conditions. We use the notation $\langle\cdot\rangle = \sqrt{1+|\cdot|^2}$.
\begin{assumption}
\label{ass1_atsuhide}
The following conditions hold:
\begin{enumerate}
\item[\rm(H.1)]
$\Psi\in C^1(0,\infty)$, $\Psi\geq0$, $\Psi'\geq0$, and $\Psi(+0)=\Psi(0)=0$.
\item[\rm(H.2)]
There exist constants $c_1,d_1\geq 0$ such that $u^{1/2} \leq c_1\Psi(u)+d_1$, for all
$u\geq 0$.
\item[\rm(H.3)]
There exist constants $c_2,d_2\geq 0$ such that $u\Psi'(u) \leq c_2\Psi(u)+d_2$, for all $u\geq 0$.
\item[\rm(H.4)]
There exists a constant $\mu>0$ such that $2 u\Psi'(u) \geq \mu \Psi(u)$, for all $u \geq 0$.
\item[\rm(H.5)]
$\langle x\rangle V$ is $H_0$-compact.
\end{enumerate}
\end{assumption}
We note that by (H.5) the operator $H$ is self-adjoint on $\Dom(H)=\Dom(H_0)$.

\begin{remark}\label{rem1_atsuhide}
{\rm
We note that (H.1)-(H.2) imply $\Spec(H_0)=[0,\infty)$. By (H.4) we find that the set $\left\{u>0 :
\Psi'(u)=0\right\}$ and $\Spec_{\rm p}(H_0)$ are empty, and that $\Spec_{\rm ac}(H_0)=[0,\infty)$ (see
\cite[Prop. 1.1]{IW}). We also find by (H.5) that $\Spec_{\rm ess}(H_0)=\Spec_{\rm ess}(H)=[0,\infty)$.}
\end{remark}

Let $U(\theta) : L^2(\RR^d)
\to L^2(\RR^d)$ be the unitary transform defined by
\begin{align*}
U(\theta)\psi(x) := e^{d\theta/2}\psi(e^\theta x), \quad \psi \in L^2(\RR^d), \; \theta\in\RR.\label{dilation}
\end{align*}
Its generator is
$A = \frac{1}{2}(x\cdot p + p \cdot x)$,
i.e., $A$ is self-adjoint and $U(\theta) = e^{i\theta A}$.

Recall that a bounded operator $B$ is called of class $C^1(A)$ if for all $\psi\in L^2(\RR^d)$ the map
$t \mapsto e^{itA}Be^{-itA}\psi$ is of class $C^1$, and a self-adjoint operator $B$ is called of class
$C^1(A)$ if for all $\psi\in L^2(\RR^d)$ and some $z\in\rho(B)$ the map $t \mapsto e^{itA}(B-z)^{-1}
e^{-itA}\psi$ is of class $C^1$, where $\rho(B)$ is the resolvent set of $B$. For more details we refer to
\cite{AmBoGe}. According to \cite{AmBoGe}, we first prove $C^1(A)$ property of $H$.

\begin{lemma}
\label{lem1_atsuhide}
Let {\rm(H.1)-(H.3)} and {\rm (H.5)} of Assumption \ref{ass1_atsuhide} hold. Then $H$ belongs to the
class $C^1(A)$.
\end{lemma}

\begin{proof}
Let $\mathscr{S}(\RR^d)$ be the Schwartz space, which is a dense subspace in both
$\Dom(H_0)$ and $\Dom(A)$. 
For a $\psi\in\mathscr{S}(\RR^d)$, Fourier transform gives
\begin{equation*}
\mathscr{F}\left[\langle H_0\rangle e^{i\theta A}\langle H_0\rangle^{-1}\psi\right](\xi)
=e^{-d\theta/2}\left\langle\Psi(\xi^2/2)\right\rangle\left\langle \Psi(e^{2\theta}\xi^2/2)\right\rangle^{-1}
\mathscr{F}[\psi](e^{-\theta}\xi).
\end{equation*}
Thus $e^{i\theta A}\Dom(H_0)\subset\Dom(H_0)$ and
\begin{equation}
\sup_{|\theta|\leq1}\left\|\langle H_0\rangle e^{i\theta A}\langle H_0\rangle^{-1}\right\|<\infty.
\label{eq1_1}
\end{equation}
hold. By a straightforward computation we obtain
\begin{equation*}
i[H_0,A]=\Psi'(p^2/2)p^2
\end{equation*}
on $\mathscr{S}(\RR^d)$. Since for all $\psi,\varphi\in\mathscr{S}(\RR^d)$, the form
\begin{equation}
\inner{\psi}{\Psi'(p^2/2)p^2\varphi}=\inner{\sqrt{\Psi'(p^2/2)}p\psi}{\sqrt{\Psi'(p^2/2)}p\varphi}\label{eq1_2}
\end{equation}
is closable and semi-bounded, there exists a self-adjoint operator associated with the closed extension of
\eqref{eq1_2}. We denote this self-adjoint operator by
$$
i[H_0,A]^0_{\mathscr{S}(\RR^d)}=\Psi'(p^2/2)p^2,
$$
and see that $\Dom(H_0)=\Dom(i[H_0,A]^0_{\mathscr{S}(\RR^d)})$ holds by (H.3) and (H.4). In virtue of \cite[Prop. II.1]{Mo}, the self-adjoint operator $i[H_0,A]^0$
associated with the form $i[H_0,A]$ on $\Dom(A)\cap\Dom(H_0)$ is well-defined and satisfies
\begin{equation*}
i[H_0,A]^0=i[H_0,A]^0_{\mathscr{S}(\RR^d)}.
\end{equation*}
Next we consider the commutator of $V$ with $A$. Since $\langle x\rangle^{-1}A\langle p\rangle^{-1}$ is bounded, there exists a constant
$C > 0$ such that
\begin{equation*}
\left\|\langle x\rangle^{-1}A\psi\right\|\leq\left\|\langle x\rangle^{-1}A\langle p\rangle^{-1}\right\|\left\|\langle p\rangle\psi\right\|\leq C\left\|\langle H_0\rangle\psi\right\|
\end{equation*}
for $\psi\in\Dom(H_0)\cap\Dom(A)$ by (H.2). Due to (H.5), for every $0<\varepsilon<1$ there exists a positive constant $C_\varepsilon$ such that
\begin{equation*}
\left\|\langle x\rangle V\varphi\right\|\leq\varepsilon\left\|H_0\varphi\right\|+C_\varepsilon\left\|\varphi\right\|.
\end{equation*}
Thus we have
\begin{equation*}
\left|\inner{V\psi}{A\varphi}\right|=\left|\inner{\langle x\rangle V\psi}{\langle x\rangle^{-1}A\varphi}\right|
\leq C\left\|\langle H_0\rangle\psi\right\|\left\|\langle H_0\rangle\varphi\right\|
\end{equation*}
and
\begin{equation*}
\left|\inner{V\psi}{A\varphi}-\inner{A\psi}{V\varphi}\right|
\leq C\left\|\langle H_0\rangle\psi\right\|\left\|\langle H_0\rangle\varphi\right\|
\end{equation*}
for $\psi, \varphi\in\Dom(A)\cap\Dom(H_0)$. Using the Riesz representation theorem, there exists a
bounded operator $\mathscr{L}: \mathscr{H}_2=\Dom(H_0)\to\mathscr{H}_{-2}=\mathscr{H}_2^*$ satisfying
\begin{equation*}
\inner{V\psi}{A\varphi}-\inner{A\psi}{V\varphi}=\inner{\psi}{\mathscr{L}\varphi}_{\mathscr{H}_{-2}},
\quad \psi, \varphi\in\Dom(A)\cap\Dom(H_0).
\end{equation*}
Note that $\mathscr{H}_{-2}$ is the completion of
$\big\{\psi\in L^2(\RR^d): \int_0^\infty\langle \lambda\rangle^{-2}
d\inner{E_{H_0}(\lambda)\psi}{\psi}<\infty\big\}$,
where $E_{H_0}$ is the spectral measure of $H_0$. Denote $\mathscr{L}=[V,A]_{-2}$. Since
\begin{equation*}
\left|\inner{\psi}{i[H_0,A]^0\varphi}\right|
\leq C\left\|\psi\right\|\left\|\langle H_0\rangle\varphi\right\|
\end{equation*}
by (H.3) and \eqref{eq1_2}, $[H_0,A]^0\psi=[H_0,A]_{-2}\psi$ holds for all $\psi\in\Dom(H_0)$. We thus obtain
\begin{equation*}
[H,A]_{-2}=[H_0,A]^0+[V,A]_{-2}
\end{equation*}
on $\Dom(H_0)$. Also, the estimate
\begin{equation}
\left|\inner{A\psi}{H\psi}-\inner{H\psi}{A\psi}\right|=\left|\inner{\psi}{[H,A]_{-2}\psi}\right|
\leq C\left\|\langle H_0\rangle\psi\right\|^2\leq C\big(\left\|H\psi\right\|^2+\left\|\psi\right\|^2\big)
\label{eq1_3}
\end{equation}
holds for $\psi\in\Dom(A)\cap\Dom(H_0)$, where we used (H.5). Note that $\langle x\rangle(H_0-z)^{-1}\langle
x\rangle^{-1}$ is bounded by (H.3). Using the resolvent formula we have
\begin{equation*}
\left(H-z\right)^{-1}=-\left(H_0-z\right)^{-1}\langle x\rangle^{-1}\langle x\rangle V\left(H-z\right)^{-1}
+\left(H_0-z\right)^{-1}
\end{equation*}
and $\langle x\rangle(H-z)^{-1}\langle x\rangle^{-1}$ is bounded for all $z\in\rho(H)$ by (H.5). This implies that
\begin{equation}
\left(H-z\right)^{-1}\Dom(\langle x\rangle)\subset\Dom(H_0)\cap\Dom(\langle x\rangle)\subset\Dom(A),
\label{eq1_4}
\end{equation}
since $A$ is closed and $2A\psi=x\cdot(p\psi)+p\cdot(x\psi)$ holds for $\psi\in\Dom(H_0)\cap\Dom(\langle x\rangle)$ by (H.2). It is seen that \eqref{eq1_4} is equivalent to
\begin{equation}
\Dom(\langle x\rangle)\subset\left\{\psi\in\dom(A): \left(H-z\right)^{-1}\psi\in\dom(A)\right\}.
\label{eq1_5}
\end{equation}
Relations \eqref{eq1_3} and \eqref{eq1_5} imply that $H\in C^1(A)$
by an application of \cite[Th. 6.2.10(a)]{AmBoGe}.
\end{proof}

\begin{theorem}\label{thm1_atsuhide}
Let {\rm(H.1)-(H.5)} of Assumption \ref{ass1_atsuhide} hold. $\Spec_{\rm p}(H)\setminus\{0\}$ is discrete
and the multiplicity of each eigenvalue is finite. The only accumulation point of $\Spec_{\rm p}(H)$ can be at zero.
\end{theorem}

\begin{proof}
Take $f\in C_0^\infty(\lambda-\delta,\lambda+\delta)$ for $\lambda>0$ and
$\delta>0$ such that $\lambda-\delta>0$. We estimate
\begin{align}
f(H)i[H,A]_{-2}f(H)
&=
f(H)\left(\Psi'(p^2/2)p^2+i[V,A]_{-2}\right)f(H)\nonumber\\
&\geq
\mu f(H)Hf(H)+f(H)\left(-\mu V+i[V,A]_{-2}\right)f(H)\nonumber\\
&\geq
\mu(\lambda-\delta)f(H)^2+K
\label{mourre_estimate1}
\end{align}
with
$$
K=-\mu f(H)Vf(H)+f(H)i[V,A]_{-2}f(H),
$$
where we used (H.4). Condition (H.5) implies that $f(H)Vf(H)$ is a compact operator. We show that $f(H)i[V,{A}]_{-2}f(H)$
is also compact. By the Helffer-Sj\"ostrand formula and \eqref{eq1_4}, $f(H)\Dom(\langle x\rangle)\subset\Dom(H_0)\cap
\Dom(\langle x\rangle)$ holds. We then obtain
\begin{equation*}
\inner{\psi}{f(H)[V,A]_{-2}f(H)\varphi}
=\inner{\langle x\rangle Vf(H)\psi}{\langle x\rangle^{-1}Af(H)\phi}-\inner{\langle x\rangle^{-1}Af(H)\psi}{\langle x\rangle Vf(H)\varphi},
\end{equation*}
for every $\psi, \varphi\in \Dom(\langle x\rangle)$.
This shows that $f(H)i[V,{A}]_{-2}f(H)$ is compact by (H.2) and (H.5). The Mourre inequality \eqref{mourre_estimate1}
then completes the proof by making use of \cite[Cor. 7.2.11]{AmBoGe}.
\end{proof}

\begin{example}\label{ex1_atsuhide}
{\rm
The function $\Psi(u)=(2u+m^{2/\alpha})^{\alpha/2}-m$ with $1\leq\alpha<2$ and $m>0$ satisfies conditions (H.1)-(H.4). Indeed,
\begin{equation*}
2u \Psi'(u)=\alpha\Psi(u) +
\alpha m\left( 1-m^{\frac{2}{\alpha}-1}(2u+m^{\frac{2}{\alpha}})^{\frac{\alpha}{2}-1} \right) \geq
\alpha\Psi(u)
\end{equation*}
and (H.3)-(H.4) hold. It can be directly verified that $\tilde{V}(x)=\langle x\rangle V(x)$ belongs to $L^q(\RR^d)$
with $q=2$ for $d<2\alpha$, and $q>\frac{d}{\alpha}$ for $d\geq 2\alpha$, satisfies (H.5) (see also \cite[Prop. 1.5]{Is}).
This includes Coulomb-like local singularities.
}
\end{example}

\begin{example}\label{ex2_atsuhide}
{\rm
Let $\Psi(u)=(2u)^{\alpha/2}$ with $\alpha\geq1$. Clearly, $2u\Psi'(u)=\alpha\Psi(u)$ holds. The singular part $\tilde{V}$
as in Example \ref{ex1_atsuhide} satisfies (H.5), which can be seen in a similar manner as in the previous example.}
\end{example}

Due to (H.2), in Examples \ref{ex1_atsuhide}-\ref{ex2_atsuhide} we needed $\alpha\geq1$. If $V$ has a
smoothness as given in (H.6) below, we can relax to $\alpha>0$ without (H.2) and (H.5) (see Example \ref{ex3_atsuhide}).

\begin{assumption}
\label{ass2_edited_by_atsuhide}
The following condition holds:
\begin{enumerate}
\item[\rm(H.6) ]
The potential satisfies $V\in C^1(\RR^d)$, and
\begin{align*}
|V(x)| \leq c \braket{x}^{-\gamma}, \quad |\nabla V(x)| \leq c \braket{x}^{-1-\gamma}
\end{align*}
with suitable constants $c, \gamma > 0$, for all $x\in\R^d$.
\end{enumerate}
\end{assumption}

Note that (H.6) implies that $V$ is bounded and $H$ is self-adjoint on $\Dom(H)=\Dom(H_0)$.

\begin{remark}\label{rem2_atsuhide}
{\rm
It is seen that $\lim_{u\rightarrow\infty}\Psi(u)=\infty$ holds by (H.4) without (H.2). Indeed,
\begin{equation*}
\left(u^{-\delta}\Psi(u)\right)'=u^{-\delta}\left(\Psi'(u)-\delta u^{-1}\Psi(u)\right)\geq u^{-1-\delta}\Psi(u)
\left(\mu/2-\delta\right)>0
\end{equation*}
for $0<\delta<\mu/2$ and $0<u_0\leq u$. Thus $\Psi(u)\geq u_0^{-\delta}\Psi(u_0)u^\delta\rightarrow\infty$ as $u\rightarrow\infty$.
Moreover, we find by (H.4) that the set $\left\{u>0 : \Psi(u)\not=0, \Psi'(u)=0\right\}$ and $\Spec_{\rm p}(H_0)\setminus\{0\}$
are empty. Under (H.6), $\Spec_{\rm ess}(H_0)=\Spec_{\rm ess}(H)$ also holds.}
\end{remark}

First we prove the counterparts of Lemma \ref{lem1_atsuhide} and Theorem \ref{thm1_atsuhide} for
the case when (H.6) holds. Without assuming (H.2) we can not use \eqref{eq1_4} directly since
$\langle p\rangle(H_0-z)^{-1}$ may not be bounded. However, making use of \cite{Mo} will yield
the stronger \eqref{eq2_5}.

\begin{lemma}\label{lem2_atsuhide}
Let {\rm (H.1)} and {\rm (H.3)} of Assumption \ref{ass1_atsuhide}, and {\rm (H.6)} of Assumption
\ref{ass2_edited_by_atsuhide} hold. Then $H$ belongs to class $C^1(A)$.
\end{lemma}

\begin{proof}
From \eqref{eq1_1} and by writing
\begin{equation*}
\langle H\rangle e^{i\theta A}\langle H\rangle^{-1}=\langle H\rangle\langle H_0\rangle^{-1}
\langle H_0\rangle e^{i\theta A}\langle H_0\rangle^{-1}\langle H_0\rangle\langle H\rangle^{-1},
\end{equation*}
we obtain
\begin{equation}
\sup_{|\theta|\leq1}\big\|\langle H\rangle e^{i\theta A}\langle H\rangle^{-1}\big\|<\infty.
\label{eq2_1}
\end{equation}
Since $V$ is bounded and differentiable, a direct computation on $\mathscr{S}(\RR^d)$
gives
\begin{equation*}
[V,A]=ix\cdot\nabla V,
\end{equation*}
which is a bounded operator due to assumption (H.6). Therefore the commutator $[H,A]^0$ can be defined by
\begin{equation}
[H,A]^0=[H_0,A]^0+ix\cdot\nabla V,
\label{eq2_4}
\end{equation}
with $\Dom([H,A]^0)=\Dom([H_0,A]^0)\supset\Dom(H_0)$. By \eqref{eq2_1}-\eqref{eq2_4} and \cite[Prop. II.2]{Mo}
we then have the inclusion
\begin{equation}
\left(H-z\right)^{-1}\Dom(A)\subset\Dom(A).
\label{eq2_5}
\end{equation}
Note also that for $\psi\in\Dom(A)\cap\Dom(H)$ the estimate
\begin{equation}
\left|\inner{A\psi}{H\psi}-\inner{A\psi}{H\psi}\right|\leq C\left(\|H\psi\|^2+\|\psi\|^2\right)
\label{eq2_6}
\end{equation}
holds. Relations \eqref{eq2_5}-\eqref{eq2_6} and \cite[Th. 6.2.10(a)]{AmBoGe}
then imply that $H$ belongs to class $C^1(A)$.
\end{proof}

\begin{theorem}\label{thm2_atsuhide}
Let {\rm (H.1)}, {\rm (H.3)} and {\rm (H.4)} of Assumption \ref{ass1_atsuhide}, and {\rm (H.6)} of Assumption
\ref{ass2_edited_by_atsuhide} hold. Then $\Spec_{\rm p}(H)\setminus\{0\}$ is discrete and the multiplicity of its elements is
at most finite. The only possible accumulation point of $\Spec_{\rm p}(H)$ is zero.
\end{theorem}

\begin{proof}
Take $f\in C_0^\infty(\lambda-\delta,\lambda+\delta)$ for $\lambda>0$ and $\delta>0$
such that $\lambda-\delta>0$. Since $(\mu V+x\cdot\nabla V)f(H)$ is a compact operator, we have the Mourre inequality
\begin{align}
f(H)i[H,A]^0 f(H)&=f(H)\left( \Psi'(p^2/2)p^2-x\cdot\nabla V\right)f(H)\nonumber\\
&\geq\mu(\lambda-\delta)f(H)^2-f(H)(\mu V+x\cdot\nabla V)f(H),
\label{mourre_estimate2}
\end{align}
where we used (H.4). Lemma \ref{lem2_atsuhide} and the Mourre inequality \eqref{mourre_estimate2} complete the proof
again by an application of \cite[Cor. 7.2.11]{AmBoGe}.
\end{proof}

\begin{example}\label{ex3_atsuhide}
{\rm
Let $\Psi(u)=(2u+m^{2/\alpha})^{\alpha/2}-m$ with $0<\alpha<2$ and $m>0$ or $\Psi(u)=(2u)^{\alpha/2}$ with $\alpha>0$. If $V$ satisfies {\rm (H.6)} of Assumption \ref{ass2_edited_by_atsuhide}, then the claim of Theorem \ref{thm2_atsuhide} holds for $H=\Psi(p^2/2)+V$.}
\end{example}

By the above arguments, we can prove absence of embedded eigenvalues above zero as follows.

\begin{theorem}{\label{MG1}}
Let conditions {\rm (H.1)}, {\rm (H.3)} and {\rm (H.4)} of Assumption \ref{ass1_atsuhide} and {\rm (H.6)} of
Assumption \ref{ass2_edited_by_atsuhide} hold, and suppose that $V$ satisfies
 \begin{align*}
        -\mu V -x\cdot \nabla V \geq 0.
 \end{align*}
Then $H$ has no positive eigenvalue.
\end{theorem}

\begin{proof}
For $f\in C_0^\infty(a,b)$ with $0<a<b<\infty$, we have the Mourre inequality
\begin{align}
     f(H)i[H,A]^0 f(H)
                          & \geq  \mu f(H) H f(H) + f(H) (-\mu V-x\cdot \nabla V) f(H) \nonumber\\
                          & \geq  \mu a f(H)^2 \label{strict_mourre}
\end{align}
without any compact operators. This yields the absence of the eigenvalues in $(a,b)$ immediately (for a
proof see \cite[Cor. 7.2.11]{AmBoGe} or \cite[Th.]{Mo}).
\end{proof}

The estimate \eqref{strict_mourre} is called the strict Mourre inequality, which we will use again below. 

\begin{example}\label{ex4_atsuhide}
{\rm
Let $\Psi(u)=(2u)^{\alpha/2}$ with $\alpha>0$, and
\begin{equation*}
V(x)=-\frac{C}{(1+|x|^2)^{\beta/2}}
\end{equation*}
with $C>0$ and
$0<\beta\leq\alpha$. Then $H=\Psi(p^2/2)+V$ satisfies the assumptions of Theorem \ref{MG1}. Indeed,
with $\mu=\alpha$ we have
\begin{equation*}
-\alpha V-x\cdot\nabla V=C\left(\alpha\langle x\rangle^{-\beta}-\beta|x|^2\langle x\rangle^{-\beta-2}\right)
=C\left((\alpha-\beta)\langle x\rangle^{-\beta}+\beta\langle x\rangle^{-\beta-2}\right)>0.
\end{equation*}
Then we have the same conclusion as in Example \ref{EX3_edited_by_atsuhide}.}
\end{example}

\begin{theorem}{\label{MG2}}
Let conditions {\rm (H.1)}, {\rm (H.3)} and {\rm (H.4)} of Assumption \ref{ass1_atsuhide} and {\rm (H.6)} of
Assumption \ref{ass2_edited_by_atsuhide} hold, and suppose that there exists a function $F: \R^d\to \R^+$
such that
 \begin{align}
       \Psi'(p^2/2)p^2 \geq F(x).    \label{inq248}
 \end{align}
If there exists $\vep>0$ such that
 \begin{align}
        -\vep \mu V + (1-\vep) F - x\cdot \nabla V \geq 0,   \label{inq249}
 \end{align}
then $H$ has no positive eigenvalue.
\end{theorem}

\begin{proof}
For $f\in C_0^\infty(a,b)$ with $0<a<b<\infty$, our argument is based on the estimates
\begin{align*}
  & f(H) \left( \Psi'(p^2/2)p^2 - x\cdot \nabla V\right)  f(H) \nonumber\\
  & \qquad =  f(H) \left( \vep \Psi'(p^2/2)p^2 + (1-\vep)\Psi'(p^2/2)p^2 - x\cdot \nabla V\right) f(H) \nonumber\\
  & \qquad \geq  f(H) \left( \vep \mu H - \vep\mu V(x) + (1-\vep)F(x) - x\cdot \nabla V\right) f(H) \nonumber\\
  & \qquad \geq  \vep \mu a f(H)^2.
\end{align*}
The strict Mourre inequality yields the absence of the eigenvalues in $(a,b)$.
\end{proof}

\begin{example}{\label{EX6}}
{\rm
Let $\Psi(u)=(2u)^{\alpha/2}$ with $0<\alpha<d$. The generalized Hardy inequality (\cite[Thm. 2.5]{H77}) says that
there exists $C_\alpha>0$ such that
\begin{align*}
      \Psi(p^2/2) \geq \frac{C_\alpha}{|x|^{\alpha}},
\end{align*}
where we used the assumption that $\alpha<d$. The condition \eqref{inq249} then becomes
\begin{align}
   -\vep \alpha V(x)  + (1-\vep)\frac{C_\alpha}{|x|^{\alpha}} - x\cdot \nabla V(x) \geq 0.
   \label{inq462}
\end{align}
Take
\begin{align*}
    V(x) = \frac{C}{\left(1+|x|^2\right)^\nu},
\end{align*}
with $C>0$ and $\nu>\alpha/2$ satisfying (H.6). Also, it is seen that \eqref{inq462} holds for sufficiently small $C>0$. We then conclude that
$H=\Psi(p^2/2)+V$ has no positive eigenvalue by Theorem \ref{MG2}.
}
\end{example}

\begin{example}\label{ex_itaru}
{\rm
Suppose that $d\geq 3$.
Let $\Psi(u)=(2u+m^{2/\alpha})^{\alpha/2}-m$ with $0<\alpha < 2$ and $m>0$.
As stated in Example \ref{ex1_atsuhide}, the function $\Psi(u)$ satisfies (H.1), (H.3) and (H.4). Next we give an example of $F(x)$ satisfying \eqref{inq248}.
As shown in Example \ref{ex1_atsuhide}, we have $2u\Psi'(u) \geq \alpha \Psi(u)$.
Hence it is enough to find a function $F(x)$ such that
\begin{align}
  \alpha \Psi(p^2/2) \geq F(x).  \label{ineq1}
\end{align}
For simplicity, we consider the case $m=1$.
We have the lower estimates
\begin{align*}
  \Psi(u)  \geq
  \begin{cases}
    (3^{\alpha/2}-1)u \qquad 0\leq u\leq 1, \\
    (3^{\alpha/2}-1)u^{\alpha/2} \qquad u>1.
  \end{cases}
\end{align*}
This gives
$$
\alpha \Psi(p^2/2) \geq \min\big\{D_{1,\alpha}p^2, D_{2,\alpha}|p|^\alpha\big\},
$$
where $D_{1,\alpha}=2^{-1}(3^{\alpha/2}-1)$, $D_{2,\alpha} = 2^{-\alpha/2}(3^{\alpha/2}-1)$.
Thus,
\begin{align*}
  (\alpha \Psi(p^2/2))^{-1} \leq
 \max\big\{ (D_{1,\alpha}p^2)^{-1}, (D_{2,\alpha}|p|^\alpha)^{-1}\big\}
 \leq D_{1,\alpha}^{-1} p^{-2} + D_{2,\alpha}^{-1} |p|^{-\alpha}.
\end{align*}
By the generalized Hardy inequality $|p|^\gamma \geq C_\gamma |x|^{-\gamma}$ and Fourier transform
we obtain
\begin{align*}
  |x|^\gamma \geq C_\gamma |p|^{-\gamma}, \qquad 0<\gamma<d.
\end{align*}
We use this inequality for the choices $\gamma=2$ and $\gamma=\alpha$. Thus we get
\begin{align*}
  (\alpha \Psi(p^2/2))^{-1} \leq D_{1,\alpha}^{-1} C_2^{-1}|x|^2 + D_{2,\alpha}^{-1} C_\alpha^{-1}|x|^\alpha
\end{align*}
see \cite[Cor. 10.12]{Schm12}. Note that here we used the condition $d\geq 3$. This implies that
\begin{align*}
    \alpha \Psi(p^2/2) \geq  (D_{1,\alpha}^{-1} C_2^{-1}|x|^2 + D_{2,\alpha}^{-1} C_\alpha^{-1}|x|^\alpha)^{-1}.
\end{align*}
Thus we arrive at the inequality \eqref{ineq1} with the lower bound
\begin{align}
  F(x) := (D_{1,\alpha}^{-1} C_2^{-1}|x|^2 + D_{2,\alpha}^{-1} C_\alpha^{-1}|x|^\alpha)^{-1}. \label{Fx}
\end{align}
Next, we discuss the condition (H.6) and \eqref{inq249}.
In the current setup, since $\mu=\alpha$, \eqref{inq249} is of the form
\begin{align}
  -\varepsilon \alpha V(x)+(1-\varepsilon)F(x) -x\cdot \nabla V(x) \geq 0. \label{313x}
\end{align}
Take the function $F$ defined in \eqref{Fx}.
Then $F$ is positive and $F(x)=O(|x|^{-2})$ as $|x|\to \infty$.
Consider the following example
\begin{align*}
  V(x) = \frac{C}{(1+|x|^2)^\nu},
\end{align*}
with $C>0$ and $\nu\leq1$. Then (H.6) clearly holds. Note that $V=O(|x|^{-2\nu})$ and $x\cdot \nabla V=O(|x|^{-2\nu})$.
Take $\varepsilon=1/2$. Then, for sufficiently small $C>0$,
\eqref{313x} holds. Therefore, by Theorem \ref{MG2}, $H=\Psi(p^2/2)+V$ has no positive eigenvalue.
}
\end{example}

\section{Extended Birman-Schwinger principle}
In this section we will show non-existence of zero eigenvalues by an extension of the classical Birman-Schwinger
principle, which we discuss first.
Consider a classical Schr\"odinger operator $H(\gamma)=H_0-\gamma V$, $H_0=-\frac{1}{2}\Delta$,
such that $V\geq 0$ is relatively form-compact with respect to $H_0$. A consequence of this is that
$\Spec_{\mathrm{ess}}(H({\gamma}))=\Spec_{\mathrm{ess}}(H_{0})=[0,\infty)$. Also, $H({\gamma})$ is
bounded from below since $-\gamma V$ is relatively form-bounded with an infinitesimally small relative
bound.

We briefly recall the classical Birman-Schwinger principle, for details see e.g. \cite{LHB}.
For $\lambda\geq0$ we have $H({\gamma})\psi=-\lambda\psi$ exactly when $(H_0+\lambda)\psi=\gamma V\psi$.
This means that $V\psi\in \Dom((H_{0}+\lambda)^{-1})$ and $\psi=\gamma(H_{0}+\lambda)^{-1}V\psi$. Since
$\psi\in \Dom(V^{1/2})$, the identity $V^{-{1/2}} V^{1/2}\psi=\gamma(H_{0}+\lambda)^{-1}V^{1/2}V^{1/2}\psi
$ gives $V^{1/2}\psi=\gamma V^{{1/2}}(H_{0}+\lambda)^{-1}V^{1/2}V^{1/2}\psi$. Thus for a $\lambda\geq0$ we
have $H({\gamma})\psi=-\lambda\psi$ exactly when
\begin{equation*}
V^{1/2}(H_{0}+\lambda)^{-1}V^{1/2}\varphi= \frac{1}{\gamma} \varphi,
\end{equation*}
where $\varphi= V^{1/2}\psi$. Define the operator
\begin{equation*}
K_{\lambda}= V^{1/2}(H_{0}+\lambda)^{-1}V^{1/2},
\end{equation*}
called Birman--Schwinger kernel.
It can be shown that $K_\lambda$ is a compact operator for $\lambda > 0$ whenever (i) $d=1$,
$V \in L^2(\R)$, (ii) $d=2$, $V\in L^p(\R^2)$, $p >1$, and for $\lambda \geq 0$ whenever (iii) $d \geq 3$,
$V\in L^{d/2}(\R^d)$; for a proof we refer to \cite[Lem. 4.144]{LHB}, see also \cite[p61]{Sim}.
Note that
\begin{equation*}
-\lambda \in \Spec(H(\gamma)) \quad \mbox{if and only if} \quad\frac{1}{\gamma} \in\Spec(K_{\lambda}).
\end{equation*}
holds. Let $\lambda >0$ and denote by $N_\lambda(V) = \# \{\mu \leq -\lambda: \mu \, \mbox{is an eigenvalue
of} \, H(\gamma)\}$. Then the Birman-Schwinger principle says that
\begin{align*}
N_{\lambda}(V) = \dim {\mathbf 1}_{[1,\infty)}(K_{\lambda})
\quad \mbox{and} \quad
N_{0}(V)=\lim_{\lambda\downarrow0}N_{\lambda}(V) \leq \dim \mathbf {1}_{[1,\infty)}(K_{0})
\end{align*}
holds. Note that $N_{0}(V)$ does not include a count of zero eigenvalues.

We give now an upgraded variant of this which also includes zero eigenvalues, and replace $H(\gamma)$ by a
non-local Schr\"odinger operator. The Birman-Schwinger principle gives
\begin{equation*}
\dim \mathbf{1}_{(-\infty,0)}(H)\leq \dim\mathbf{1}_{[1,\infty)}(K_{0}).
\end{equation*}
while the extended Birman-Schwinger principle can be stated as
\begin{gather}
\label{bssb}
\dim\mathbf{1}_{(-\infty,0]}(H) \leq \dim \mathbf{1}_{[1,\infty)}(K_{0}),
\end{gather}
so that the left-hand side in \eqref{bssb} includes the number of non-positive eigenvalues of $H$. In particular,
whenever $\dim\mathbf{1}_{[1,\infty)} (K_{0})<1$, the implication is that $H$ has no non-positive eigenvalues.
In the case of the classical Laplacian as the kinetic energy operator, the extended Birman-Schwinger principle
is covered in \cite{Sim}, while the proof there holds in a larger framework, including the one in the present paper.
An extended Birman-Schwinger principle has been studied also in \cite{MV,T18} in another context, but since it
remained little known, we provide here a proof, following the main lines of a previous discussion in the monograph
\cite{LHB}. We discuss it for non-local Schr\"odinger operators with Bernstein functions of the Laplacian, and
leave further extensions to the reader.
Let $\Phi$ be a Bernstein function as given \eqref{bernrep}, with $b=0$, and consider $H_0=\Phi(-\Delta)$. Define
\begin{gather*}
\bar{N}_\lambda(V) = \dim \mathbf{1}_{(-\infty,-\lambda]}(H) =
\#\{\mbox{eigenvalues of }H\leq-\lambda\},\quad \lambda\geq0.
\end{gather*}
Note that $\bar{N}_{\lambda}(V)=N_{\lambda}(V)$ for $-\lambda<0$, and $\bar{N}_{0}(V) \geq N_{0}(V)$. The
extended Birman--Schwinger principle is as follows.

\begin{proposition}
\label{eBSP}%
Suppose that $V\leq0$ and it is $H_{0}$-form-compact. Let $\lambda > 0$. Then
\begin{align*}
\bar{N}_\lambda(V) = \dim \mathbf{1}_{[1,\infty)}(K_\lambda) \quad  \mbox{and} \quad
\bar{N}_{0}(V) \leq \dim \mathbf{1}_{[1,\infty)}(K_{0}).
\end{align*}
In particular, if $\dim\mathbf{1}_{[1,\infty)}(K_{0})=0$, then $H$ has no non-positive eigenvalues.
\end{proposition}

\begin{proof}
The proof of the first part follows from Lemmas \ref{00}-\ref{teranishi} below by choosing $T=H_{0}$, $S=V$, and
$L_{0}=K_0$. The second part follows by having that $\bar{N}_{0}(V) \leq \dim\mathbf{1}_{[1,\infty)}(K_{0})=0$,
and thus $\Spec_{\mathrm{ess}}(H)=[0,\infty)$.
\end{proof}

\begin{corollary}
\label{coreBS}
Let $d \geq 3$ and suppose that $V \leq 0$ is $H_0$-form compact. If $\|K_0\| < 1$, then $H$ has no non-positive
eigenvalues.
\end{corollary}
\begin{proof}
By Proposition \ref{eBSP} we have $\bar{N}_0(V)\leq \dim\mathbf{1}_{[1,\infty)}(K_{0})=0$, and
$\Spec_{\mathrm{ess}}(H)=[0,\infty)$.
\end{proof}

We show the two lemmas referred to above.
\begin{lemma}
\label{00}
Let $\mathcal {K}$ be a separable Hilbert space over ${\mathbb{C}}$, and $T$ a self-adjoint operator on
$\mathcal{K}$ such that $\inf\operatorname{Spec}(T)=0$ and $0\not\in\operatorname{Spec}_{\mathrm{p}}(T)$.
Let $S$ be a self-adjoint operator such that $S\leq0$ and $L_{\lambda}=|S|^{1/2}(T+\lambda)^{-1}|S|^{1/2}$
is compact for $\lambda\geq0$. Take $\gamma>1$. Then $\dim\mathbf{1}_{(-\infty,0]}(T+S) \leq
\dim\mathbf{1}_{(-\infty,0)}(T+\gamma S)$.
\end{lemma}

\begin{proof}
Suppose, to the contrary, that
$\dim \mathbf{1}_{(-\infty,0)}(T+\gamma S) < \dim\mathbf{1}_{(-\infty,0]}(T+S)$.
There exists a positive integer $n$ such that $\dim\mathbf{1}_{(-\infty,0)}(T+\gamma S)=n\geq0$. Then
$\dim\mathbf{1}_{(-\infty,0]}(T+S)\geq n+1$. Let $\mathcal{K}\supset\{\phi_{1},\ldots, \phi_{n+1}\}$ be
an arbitrary orthonormal system of $1_{(-\infty,0]}(T+S)$, denote by $\mathscr{L}$ the linear span of
$\{\phi_{1},\ldots,\phi_{n+1}\}$, and consider a projection $P:\mathcal{K} \to\mathscr{L}$. The restriction
$A=P(T+\gamma S) P\lceil_{\mathscr{L}}$
is a self-adjoint operator from $\mathscr{L}$ to itself. Denote the eigenvalues of $A$ by $\tilde{\mu}_{1}
\leq\tilde{\mu}_{2}\leq\ldots\leq\tilde{\mu}_{n+1}$, and write
\begin{equation*}
\mu_{j}=\sup_{\psi_{1},\ldots,\psi_{j-1}\in\mathscr{K}}\inf_{{\psi\in \Dom(A);\|\psi\|=1
\atop\psi\in\{\psi_{1},\ldots, \psi_{j-1}\}^{\perp}}} \langle\psi, A\psi\rangle
\end{equation*}
for the eigenvalues of $A$. Using the the Rayleigh-Ritz principle, they compare like $\mu_{j}\leq\tilde{\mu}_{j}$,
$j=1,\ldots, n+1$. We have $\langle f, (T+\gamma S) f\rangle=\langle f,(T+S)f\rangle+(\gamma-1)\langle f,Sf\rangle
\leq0$ for all $f\in\mathscr{L}$. Suppose that $0=\langle g, (T+\gamma S) g\rangle$ for some $g\in\mathscr{L}$. Then
$\langle g,(T+S)g\rangle=(\gamma-1)\langle g,Sg\rangle =0$ and $\langle g, Tg\rangle =0$. However, $\Spec_{\mathrm{p}}
(T) \not\ni0$ implies that $g=0$, thus $A$ is strictly negative and $\tilde{\mu}_{j}<0$, $j=1,\ldots,n+1$. In particular,
$\mu_{n+1}\leq\tilde{\mu}_{n+1}<0$. This gives $\dim\mathbf{1}_{(- \infty,0)} (T+\gamma S)\geq n+1$, contradicting that
$\dim\mathbf{1}_{(-\infty,0)}(T+\gamma S)=n$.
\end{proof}

\begin{lemma}
\label{teranishi}
If $L_{0}$ is compact, then
$\dim\mathbf{1}_{(-\infty,0]}(T+S)\leq\dim\mathbf{1}_{[1,\infty)}(L_{0})$.
\end{lemma}
\begin{proof}
By Lemma \ref{00} we have
\begin{gather}
\label{ooo}
\dim\mathbf{1}_{(-\infty,0]}(T+S)\leq\lim_{\gamma\uparrow1}
\dim\mathbf{1}_{(-\infty,0)}(T+\gamma S).
\end{gather}
Note that $\dim\mathbf{1}_{(-\infty,0)}(T+\gamma S)= \lim_{\lambda\uparrow0}\dim\mathbf{1}_{(-\infty,\lambda]}
(T+\gamma S)$, and by the Birman--Schwinger principle we have $\dim\mathbf{1}_{(-\infty, \lambda]}({T+\gamma S})
\leq\dim\mathbf{1}_{[1,\infty)}(\gamma L_{\lambda})$ for all $\lambda < 0$. Since $L_{\lambda}\leq L_{0}$, we
get $\dim\mathbf{1}_{[1,\infty)}(\gamma L_{\lambda})\leq\dim\mathbf{1}_{[1,\infty)}(\gamma L_{0})$. Combining
this with \eqref{ooo} gives
\begin{equation*}
\dim\mathbf{1}_{(-\infty,0]}(T+S) \leq\lim_{\gamma\uparrow1}\dim\mathbf{1}_{[1,\infty)}(\gamma L_{0})\leq
\dim\mathbf{1}_{[1,\infty)}(L_{0}).
\end{equation*}
\end{proof}

Now we apply the extended Birman-Schwinger principle to prove that zero is not an eigenvalue to some non-local
Schr\"odinger operators. Choose $\Phi(u)=(2u)^{\alpha/2}$ with $0<\alpha<2$, so that $H_0=\Phi(p^2/2)=
(-\Delta)^{\alpha/2}$.

\begin{theorem}\label{thm4_atsuhide}
Let $d\geq3$ and $V:\mathbb{R}^d\rightarrow\mathbb{R}$ be a bounded and compactly supported potential. Denote by $B_R(x)$ a ball of radius $R$ centered in $x$, and choose $R>0$ such that $\supp V\subset B_R(0)$. Suppose that $V\leq0$ and
\begin{equation}
\|V\|_\infty<\frac{d(2^\alpha-1)}{C_{d,\alpha}\ome_d2\sqrt{2}\sqrt{5^d}(2R)^\alpha},\label{eq4_0}
\end{equation}
where $C_{d,\alpha}=\frac{\Gamma((d-\alpha)/2)}{\pi^{d/2}2^\alpha\Gamma(\alpha/2)}$ and $\ome_d$ is the surface area of the $d$-dimensional sphere. Then $H=H_0+V$ with $H_0=(-\Delta)^{\alpha/2}$
has no non-positive eigenvalue. In particular, zero is not an eigenvalue of $H$.
\end{theorem}

\begin{proof}
Clearly, $V$ is $H_0$-compact. By Fourier transform we get
\begin{equation*}
\Phi(p^2/2)^{-1}|V|^{1/2}\psi(x)=\mathscr{F}^*|\xi|^{-\alpha}\mathscr{F}|V|^{1/2}\psi(x)
=C_{d,\alpha}\int_{B_R(0)}\frac{|V(y)|^{1/2}\psi(y)}{|x-y|^{d-\alpha}}dy
\end{equation*}
for a.e. $x\in\mathbb{R}^d$ and every $\psi\in L^2(\RR^d)$, see \cite[Sect. V.1]{St}. We estimate
\begin{align}
\||V|^{1/2}\Phi(p^2/2)^{-1}|V|^{1/2}\psi\|^2
&=
C_{d,\alpha}^2\int_{B_R(0)}\Big|\int_{B_R(0)}\frac{|V(y)|^{1/2}\psi(y)}{|x-y|^{d-\alpha}}dy\Big|^2 |V(x)| dx
\nonumber\\
&\leq
C_{d,\alpha}^2\int_{B_R(0)}\Big(\int_{B_{2R}(x)}\frac{|V(y)|^{1/2}|\psi(y)|}{|x-y|^{d-\alpha}}dy\Big)^2 |V(x)|dx
\nonumber\\
&\leq
C_{d,\alpha}^2\|V\|_{\infty}^2\int_{\mathbb{R}^d}\Big(\int_{B_{2R}(x)}\frac{|\psi(y)|}{|x-y|^{d-\alpha}}dy\Big)^2dx
\label{eq4_1}
\end{align}
Using the Hardy-Littlewood maximal function, we have
\begin{equation*}
\int_{B_{2R}(x)}\frac{|\psi(y)|}{|x-y|^{d-\alpha}}dy\leq \frac{\ome_d}{d(2^\alpha-1)}(2R)^\alpha M[\psi](x),
\end{equation*}
see \cite[Lem. (a) with $A=\ome_d/d$]{He}, where
\begin{equation*}
M[\psi](x)=\sup_{r>0}\frac{1}{|B_r(0)|}\int_{B_r(x)}|\psi(y)|dy.
\end{equation*}
By the Hardy-Littlewood-Winner theorem \cite[Th. I.1]{St} we also have
$\|M[\psi]\| \leq  2\sqrt{2}\sqrt{5^d}\|\psi\|$,
thus \eqref{eq4_1} gives

\begin{equation*}
\||V|^{1/2}\Phi(p^2/2)^{-1}|V|^{1/2}\psi\| \leq
\frac{C_{d,\alpha}\ome_d2\sqrt{2}\sqrt{5^d}(2R)^\alpha}{d(2^\alpha-1)}\|V\|_{\infty}\|\psi\|,
\end{equation*}
and thus \eqref{eq4_0} implies that $\||V|^{1/2}\Phi(p^2/2)^{-1}|V|^{1/2}\|<1$. Then the extended Birman-Schwinger principle and Corollary \ref{coreBS} complete the proof.
\end{proof}

The argument can further be extended to fully supported potentials.
\begin{corollary}
\label{cor1_atsuhide}
Let $d\geq3$ and $V \in L^\infty(\mathbb{R}^d)\cap L^1(\mathbb{R}^d)$ satisfy $V(x)\to0$ as $|x|\to\infty$. Suppose
that $\|V\|_\infty$ and $\|V\|_1$ are small enough such that
\begin{equation}
2C_{d,\alpha}^2\left[\frac{8 \cdot 5^d \ome_d^2}{d^2(2^\alpha-1)^2}\|V\|_\infty^2+\|V\|_1^2\right]<1\label{eq4_2}
\end{equation}
holds. Then $H=H_0+V$ with $H_0=(-\Delta)^{\alpha/2}$ has no non-positive eigenvalue,
in particular, zero is not an eigenvalue of $H$.
\end{corollary}

\begin{proof}
Since $V(x)\to0$ as $|x|\to\infty$, $V$ is $H_0$-compact. As in \eqref{eq4_1}, we compute
\begin{gather}
\||V|^{1/2}\Phi(p^2/2)^{-1}|V|^{1/2}\psi\|^2
\leq
C_{d,\alpha}^2\int_{\mathbb{R}^d}
\left(\Big(\int_{B_1(x)}+\int_{\mathbb{R}^d\setminus B_1(x)}\Big)
\frac{|V(y)|^{1/2}|\psi(y)|}{|x-y|^{d-\alpha}}dy\right)^2 |V(x)| dx \label{eq4_3}\\
\leq 2C_{d,\alpha}^2(I_1+I_2).\nonumber
\end{gather}
with
\begin{gather*}
I_1
=\|V\|_{\infty}^2\int_{\mathbb{R}^d}\left(\int_{B_1(x)}
\frac{|\psi(y)|}{|x-y|^{d-\alpha}}dy\right)^2dx,\\
I_2
=\int_{\mathbb{R}^d}|V(x)|\left(\int_{\mathbb{R}^d\setminus B_1(x)}
\frac{|V(y)|^{1/2}|\psi(y)|}{|x-y|^{d-\alpha}}dy\right)^2dx.
\end{gather*}
$I_1$ can be estimated by the theorem above as
\begin{equation*}
I_1\leq\left(\frac{\ome_d2\sqrt{2}\sqrt{5^d}}{d(2^\alpha-1)}\right)^2\|V\|_\infty^2,
\end{equation*}
and Schwarz inequality
gives
\begin{equation*}
I_2\leq\int_{\mathbb{R}^d}|V(x)|dx
\left(\int_{\mathbb{R}^d}|V(y)|^{1/2}|\psi(y)|dy\right)^2\leq\|V\|_1^2\|\psi\|^2.
\end{equation*}
We thus have $\||V|^{1/2}\Psi(p^2/2)^{-1}|V|^{1/2}\|<1$ by using \eqref{eq4_2}, and the extended
Birman-Schwinger principle and Corollary \ref{coreBS} complete the proof.
\end{proof}

Next we consider the massive relativistic case $\Phi(u)=(2u+m^{2/\alpha})^{\alpha/2}-m$ with $1<\alpha< 2$
and $m>0$ in dimension $d=3$. By a combination of the limiting resolvent method and the extended Birman-Schwinger
principle we prove the absence of non-positive eigenvalues of $H=H_0+V$ with $H_0=(-\Delta+m^{2/\alpha})^{\alpha/2}
-m$. We start with the following relative compactness result, which is a counterpart to our case of e.g.
\cite[Lem. 5.9]{S12} for $H_0=-\Delta$.

\begin{lemma}
\label{lem3_atsuhide}
Let $2\leq q<\infty$, $q>\frac{3}{\alpha}$ and suppose that $V\in L^q(\RR^3)$. Then $V$ is $H_0$-compact.
\end{lemma}

\begin{proof}
We first prove that $f(x)g(p)$ is a bounded operator for $f,g\in L^q(\RR^3)$ and there exists $C>0$ such that
\begin{equation}
\|f(x)g(p)\|\leq C\|f\|_q\|g\|_q.\label{fg}
\end{equation}
Put $q_1=2q/(2+q)$ and $q_2=q_1/(q_1-1)$. By the Hausdorff-Young and H\"older inequalities,
\begin{equation*}
\|g(p)\psi\|_{q_2}\leq(2\pi)^{3(\frac{1}{2}-\frac{1}{q_2})}\|g\hat{\psi}\|_{q_1}\leq(2\pi)^{\frac{1}{2}-\frac{1}{q_2}}\|g\|_q\|\psi\|
\end{equation*}
for $\psi\in C_0^\infty(\RR^3)$. This implies \eqref{fg} since $\|f(x)g(p)\psi\|\leq\|f\|_q\|g(p)\psi\|_{q_2}$
holds by using the H\"older inequality again. Let $\rho>0$ and denote by $B_\rho$ the ball of radius $\rho$
centered in the origin. Define $V_\rho(x)=V(x)\mathbf{1}_{B_\rho}(x)$ and $R_\rho(\xi)=(\Phi(\xi^2/2)+1)^{-1}
\mathbf{1}_{B_\rho}(\xi)$. Note that $V_\rho, R_\rho\in L^2(\RR^3)$. Since
$\int_{\RR^3}\int_{\RR^3}|V_\rho(x)R_\rho(\xi)|^2dxd\xi=\|V_\rho\|^2\|R_\rho\|^2$,
the operator $V_\rho(x)R_\rho(p)$ is of Hilbert-Schmidt type. Note that $(\Phi(\xi^2/2)+1)^{-1}\in L^q(\RR^3)$
due to $q>\frac{3}{\alpha}$. Thus we obtain
\begin{align*}
&\|V(H_0+1)^{-1}-V_\rho R_\rho(p)\|\\
&\quad\leq C(\|V-V_\rho\|_q\|(\Phi(\xi^2/2)+1)^{-1}\|_q+\|V\|_q\|(\Phi(\xi^2/2)+1)^{-1}-R_\rho\|_q)\to0
\end{align*}
as $\rho\to\infty$ using \eqref{fg} and the dominated convergence theorem.
\end{proof}

\begin{lemma}
\label{lem4_atsuhide}
Denote the resolvent of $H_0$ by $R(z)=(\Phi(p^2/2)-z)^{-1}$. For $s>1$,
\begin{equation}
\langle x\rangle^{-s}\Phi(p^2/2)^{-1}\langle x\rangle^{-s}=\langle x\rangle^{-s}R(\pm i0)\label{resolvent_limits}
\langle x\rangle^{-s}
\end{equation}
exists in the $L^2$ operator norm sense.
\end{lemma}

\begin{proof}
For the case of $\alpha=2$, the existence of \eqref{resolvent_limits} is well known (see, e.g., \cite[Th. 5.1]{STK}). For $z\in\mathbb{C}\setminus\RR$,
fix $0<\theta=\arg(z+m)<2\pi$. Then $(z+m)^{2/\alpha}=|z+m|^{2/\alpha}e^{2i\theta/\alpha}e^{4in\pi/\alpha}$ with $n\in
\mathbb{Z}$. Since $\alpha<2$, the solution of $(\zeta+m^{2/\alpha})^{\alpha/2}-m-z=0$ for $\zeta$ is unique in the
complex plane by the fixed $n\in\mathbb{Z}$. We write $\zeta(z)=(z+m)^{2/\alpha}-m^{2/\alpha}$ with ${\rm Im}\sqrt{\zeta(z)}
>0$. We first show that the operator $\langle x\rangle^{-s}R(z)\langle x\rangle^{-s}$ has the integral kernel
\begin{equation}
\mathscr{K}(x,y;z)=\langle x\rangle^{-s}\frac{e^{i\sqrt{\zeta(z)}|x-y|}}{2\alpha\pi\left(\zeta(z)+
m^{2/\alpha}\right)^{\alpha/2-1}|x-y|}\langle y\rangle^{-s}.
\label{kernel}
\end{equation}
The kernel of $R(z)$ is given by the Fourier transform
\begin{equation*}
\frac{1}{(2\pi)^3}\int_{\mathbb{R}^3}\frac{e^{i(x-y)\cdot\xi}}{\Phi(|\xi|^2/2)-z}d\xi,
\end{equation*}
which is radially symmetric in $x-y$. It is sufficient to consider $x=(0,0,x_3)$ with $x_3>0$ and
$y=(0,0,0)$ as in \cite[Ex. 1, XI.7]{RS2} or \cite[Ex. 6.1]{S12} for $\alpha=2$. Using polar coordinates, we obtain
\begin{eqnarray*}
\int_{\mathbb{R}^3}\frac{e^{i(x-y)\cdot\xi}}{\Phi(|\xi|^2/2)-z}d\xi
&=&
2\pi\int_0^\infty\int_{-\pi/2}^{\pi/2}\frac{e^{ix_3r\sin\phi}r^2\cos\phi}{\Phi(r^2/2)-z}d\phi dr\nonumber\\
&=&
2\pi\int_0^\infty\int_{-1}^1\frac{e^{ix_3r\lambda}r^2}{\Phi(r^2/2)-z}d\lambda dr
= \frac{2\pi}{ix_3}\int_{-\infty}^\infty\frac{re^{ix_3r}}{\Phi(r^2/2)-z}dr,
\end{eqnarray*}
where we made the change of variable $\lambda=\sin\phi$. The improper integral can be computed by an
application of the residue theorem. Since $x_3>0$ and $\alpha>1$, the integral on the upper half circle with
radius $\rho>0$ vanishes as $\rho\rightarrow\infty$. Since
\begin{equation*}
\Phi(r^2/2)-z
=\frac{\alpha}{2}\left(\zeta(z)+m^{2/\alpha}\right)^{\alpha/2-1}(r^2-\zeta(z)) + O\left((r^2-\zeta(z)^2\right)
\end{equation*}
as $r\rightarrow\sqrt{\zeta(z)}$, we have
\begin{equation*}
\mathop{\rm Res}_{r=\sqrt{\zeta(z)}}\frac{re^{ix_3r}}{\Phi(r^2/2)-z}
=\frac{e^{ix_3\sqrt{\zeta(z)}}}{\alpha\left(\zeta(z)+m^{2/\alpha}\right)^{\alpha/2-1}}.
\end{equation*}
We thus obtain \eqref{kernel}. We can assume $1<s<3/2$ without loss of generality and next prove
\begin{equation}
\int_{\mathbb{R}^3}|x-y|^{-2}\langle y\rangle^{-2s}dy\leq C(|x|^{1-2s}+|x|^{1-2s'}),
\label{eq5_1}
\end{equation}
where $1<s'<s$. When $|y|\leq|x|/2$, we have $|x-y|\geq|x|-|y|\geq|x|/2$. We estimate
\begin{equation*}
\int_{|y|\leq|x|/2}|x-y|^{-2}\langle y\rangle^{-2s}dy\leq16\pi|x|^{-2}\int_0^{|x|/2}r^{2-2s}dr
\leq C|x|^{1-2s}.
\end{equation*}
When $|y|>|x|/2$, we have $|x-y|\leq|x|+|y|<3|y|$ and obtain
\begin{equation*}
\int_{|y|>|x|/2}|x-y|^{-2}\langle y\rangle^{-2s}dy\leq\langle x/2\rangle^{1-2s'}
\int_{\RR^3} |x-y|^{-2}\langle x-y\rangle^{-1-2(s-s')}dy\leq C|x|^{1-2s'}.
\end{equation*}
A combination of the above two estimates then gives \eqref{eq5_1}. It follows from \eqref{eq5_1} that
\begin{equation}
\int_{\mathbb{R}^3}\int_{\mathbb{R}^3}|\mathscr{K}(x,y;z)|^2dxdy
\leq C\int_{\mathbb{R}^3}\langle x\rangle^{-2s}(|x|^{1-2s}+|x|^{1-2s'})dx<\infty.
\label{hilbert_schmidt}
\end{equation}
Therefore the kernel $\mathscr{K}$ has the Hilbert-Schmidt property, including the point $z=0$. In particular,
the constant $C$ in the right-hand-side of \eqref{hilbert_schmidt} is independent of $z$.  By dominated
convergence we then have
$\int_{\mathbb{R}^3}\int_{\mathbb{R}^3}|\mathscr{K}(x,y;z)-\mathscr{K}(x,y;0)|^2dxdy\to 0$
as $z\rightarrow0$. This means that \eqref{resolvent_limits} exists in the $L^2$ operator norm sense.
\end{proof}

\begin{theorem}\label{thm5_atsuhide}
Let $V:\mathbb{R}^3\rightarrow\mathbb{R}$, $V\leq0$, such that
\begin{equation}
\|\langle x\rangle^s|V|^{1/2}\|_{\infty}< \|\langle x\rangle^{-s}R(\pm i0)\langle x\rangle^{-s}\|^{1/2}\label{eq5_2}
\end{equation}
for $s>1$. Then $H=H_0+V$ has no non-positive eigenvalue, in particular, zero is not an eigenvalue of $H$.
\end{theorem}

\begin{proof}
Since
\begin{equation*}
\langle x\rangle^{2sq}|V(x)|^q\leq\|\langle x\rangle^s|V|^{1/2}\|_{\infty}^{2q}
\end{equation*}
and $s>1$, $V\in L^q(\RR^3)$ for any $q>3/2$. Therefore, $V$ is $H_0$ -compact by Lemma \ref{lem3_atsuhide}.
We obtain
\begin{equation*}
\||V|^{1/2}\Phi(p^2/2)^{-1}|V|^{1/2}\|\leq\||V|^{1/2}\langle x\rangle^s\|_\infty \|\langle x\rangle^{-s}R(\pm i0)
\langle x\rangle^{-s}\| \|\langle x\rangle^s|V|^{1/2}\|_\infty<1
\end{equation*}
by \eqref{eq5_2}.
Then the extended Birman-Schwinger principle and Corollary \ref{coreBS} complete the proof.
\end{proof}

\begin{example}\label{ex5_atsuhide}
{\rm
Let $H=(-\Delta+m^{2/\alpha})^{\alpha/2}-m+V$ with $1<\alpha<2$ and $m>0$, and let
\begin{equation*}
V(x)=-\frac{C}{(1+|x|^2)^\nu}
\end{equation*}
with $\nu>1$ and $C>0$ sufficiently small that satisfy the conclusion of Example \ref{ex_itaru}. In addition,
if $C>0$ satisfies $0<C<\|\langle x\rangle^{-\nu}R(\pm i0)\langle x\rangle^{-\nu}\|^{1/2}$, then $\Spec_{\rm p}(H)
= \emptyset$ in virtue of Example \ref{ex_itaru} and Theorem \ref{thm5_atsuhide}.
}
\end{example}

\section{Existence vs absence of at-edge eigenvalues in function of the coupling constant}

Finally we present a first case where by a combination of the above results we can prove the existence/absence
of zero-eigenvalue dichotomy for fractional Schr\"odinger operators in function of the magnitude of the coupling
constant. Due to the simple form of the potential this can be seen as a prototype case, and further examples can
be constructed in a similar way.

As already mentioned in the Introduction, in \cite{JL18} we have obtained the following result on the existence of embedded
eigenvalues at zero. Let $P$ be a harmonic polynomial, homogeneous of degree $l \geq 0$, i.e., such that $P(kx) = k^l P(x)$
for all $k>0$, and $\Delta P = 0$. Also, let $a, b, c \in \mathbb C$, with $c \notin \mathbb Z_{\leq 0}$, and consider Gauss'
hypergeometric function
\begin{equation*}
_2\textbf{F}_1\left( \left. \begin{array}{c}
				a \quad b \\ c
			\end{array}  \right| z \right)
= \sum_{k=0}^{\infty} \frac{(a)_{k} (b)_{k} }{\Gamma(c+k)} \cdot \frac{z^k}{k !}, \quad z \in \mathbb C,
\end{equation*}
with the Pochhammer symbol $(a)_k = a(a+1) \cdots (a+k-1)$ and the usual Gamma function $\Gamma$. Let $\kappa>0$. Then
with $\delta = d+2l$, $d\geq 1$, the family of decaying potentials
\begin{equation}
\label{exist}
V_{\kappa,\alpha}(x) =  -\frac{2^\alpha}{\Gamma(\kappa)} \Gamma\left(  \frac{\delta+\alpha}{2}\right)
\Gamma\left( \frac{\alpha}{2} + \kappa \right)(1+|x|^2)^\kappa \,_2\textbf{F}_1\left( \left. \begin{array}{c}
				\frac{\delta+\alpha}{2} \quad \frac{\alpha}{2} + \kappa \\ \frac{\delta}{2}
			\end{array}  \right| -|x|^2 \right),
\end{equation}
generate zero-eigenvalues with the eigenfunctions
\begin{equation}
\label{exist2}
\varphi_\kappa(x) = \frac{P(x)}{(1+|x|^2)^{\kappa}},
\end{equation}
i.e.,
\begin{equation}
\label{exist3}
(-\Delta)^{\alpha/2} \varphi_\kappa + V_{\kappa,\alpha} \varphi_\kappa = 0
\end{equation}
holds in distributional sense, for all $0 < \alpha < 2$. Next we will study a sub-family of these potentials.

\begin{theorem}
\label{dicho}
Let $d=3$ and consider the potential
$$
V_\alpha(x) = -\frac{C_\alpha}{\left(1+|x|^2\right)^{\alpha}}
$$
where $C_\alpha > 0$, and the non-local Schr\"odinger operator $H_\alpha = (-\Delta)^{\alpha/2} + V_\alpha$ on
$L^2(\R^d)$. The following occur:
\begin{enumerate}
\item[(1)]
If $\frac{3}{2} <\alpha<2$, then there exist $A(\alpha)>0$ such that $(-\infty,0]\cap\Spec_{\rm p}(H_\alpha)
$ is empty whenever $C_\alpha<A(\alpha)$, in particular, $0 \not \in \Spec_{\rm p}(H_\alpha)$.
Moreover, $A(\alpha)$ satisfies
\begin{equation}
A(\alpha)\geq\frac{1}{4\pi C_{3,\alpha}I_\alpha}\sqrt{\frac{\alpha}{6}
\left(\frac{3(2^d-1)}{10\sqrt{10}}I_\alpha\right)^{2-\frac{2\alpha}{3}}\left(\frac{3}{\alpha}-1\right)^{1-\frac{\alpha}{3}}},
\label{dicho_1}
\end{equation}
where $C_{3,\alpha}=\frac{\Gamma((3-\alpha)/2)}{\pi^{3/2}2^\alpha\Gamma(\alpha/2)}$ and $I_\alpha=\int_0^\infty r^2\langle r\rangle^{-2\alpha}dr$.

\item[(2)]
If $0 < \alpha < 2$ and $C_\alpha = 2^{\alpha} \frac{\Gamma\left(  \frac{\delta+\alpha}{2}\right)}
{\Gamma\left(\frac{\delta-\alpha}{2}\right)}$ with $\delta = 3 + 2\ell$ for $\ell \in \{0\}\cup \mathbb N$, then
$0 \in \Spec_{\rm p}(H_\alpha)$.
\end{enumerate}
Furthermore, we have the following properties:
\begin{enumerate}
\item[(3)]
If $0<\alpha<2$ and $C_\alpha<a$ for $a>0$, then $\Spec_{\rm p}(H_\alpha) \cap (a,\infty) = \emptyset$. Moreover,
if $\frac{3}{2} <\alpha<2$ and $C_\alpha< \min\{a,A(\alpha)\}$, where $A(\alpha)$ is chosen as in (1) above, then
$\Spec_{\rm p}(H_\alpha)\subset(0,a]$.

\item[(4)]
If $0 < \alpha < 2$, then there exists a constant $B(\alpha) > 0$ such that
$\Spec_{\rm d}(H_\alpha)$ is not empty whenever $C_\alpha > B(\alpha)$.
\end{enumerate}
\end{theorem}

\begin{proof}
(1) For $\rho>0$, it follows from
\begin{equation*}
\|\langle x\rangle^{-\alpha}\Phi(p^2/2)^{-1}\langle x\rangle^{-\alpha}\psi\|^2\leq 32\pi^2C_{3,\alpha}^2
\left(\frac{1000}{9(2^d-1)^2}\rho^{2\alpha}+I_\alpha^2\rho^{2\alpha-6}\right)\|\psi\|^2.
\end{equation*}
that $\langle x\rangle^{-\alpha}\Phi(p^2/2)^{-1}\langle x\rangle^{-\alpha}$ is bounded by replacing $B_1(x)$ with
$B_\rho(x)$ of \eqref{eq4_3} in the proof of Corollary \ref{cor1_atsuhide}. Define $A(\alpha)=
1/\|\langle x\rangle^{-\alpha}\Phi(p^2/2)^{-1}\langle x\rangle^{-\alpha}\|$. Then
\begin{equation*}
\||V_\alpha|^{1/2}\Phi(p^2/2)^{-1}|V_\alpha|^{1/2}\|=C_\alpha\|\langle x\rangle^{-\alpha}\Phi(p^2/2)^{-1}\langle x\rangle^{-\alpha}\|<1
\end{equation*}
holds for $C_\alpha<A(\alpha)$. To estimate $A(\alpha)$, note that the minimiser of the function
$f(\rho)=1000\rho^{2\alpha}/(9(2^d-1)^2)+I_\alpha^2\rho^{2\alpha-6}$ is
\begin{equation*}
\rho_{\rm min}=\left(\frac{3(2^d-1)}{10\sqrt{10}}I_\alpha\right)^{\frac{1}{3}}\left(\frac{3}{\alpha}-1\right)^{\frac{1}{3}},
\end{equation*}
so that we have
\begin{equation*}
f(\rho_{\rm min})=\frac{3I_\alpha^2}{\alpha}\left(\frac{3(2^d-1)}{10\sqrt{10}}
I_\alpha\right)^{\frac{2\alpha}{3}-2}\left(\frac{3}{\alpha}-1\right)^{\frac{\alpha}{3}-1}.
\end{equation*}
This implies \eqref{dicho_1}.

\medskip
\noindent
(2)  Choosing $\kappa = \frac{\delta-\alpha}{2}$ in \eqref{exist}, we have
\begin{eqnarray*}
\lefteqn{
V_{(\delta-\alpha)/2,\alpha}(x) } \\
&=&
-\frac{2^\alpha}{\Gamma\left(\frac{\delta-\alpha}{2}\right)}
\Gamma\left(\frac{\delta+\alpha}{2}\right)\Gamma\left(\frac{\alpha}{2}+\frac{\delta-\alpha}{2}\right)
(1+|x|^2)^{\frac{\delta-\alpha}{2}}
\,_2\textbf{F}_1\left( \left.
\begin{array}{c} \frac{\delta+\alpha}{2} \,,\, \frac{\alpha}{2}+\frac{\delta-\alpha}{2} \\ \frac{\delta}{2}
\end{array} \right| -|x|^2 \right).
\end{eqnarray*}
Using the identity
\begin{equation*}
\frac{(1-z)^{-a}}{\Gamma(b)} =\,
_2\textbf{F}_1\left( \left. \begin{array}{c} a \,,\, b \\ b \end{array} \right| z \right)
\end{equation*}
see \cite[eq. (15.4.6)]{NIST}  we readily obtain
$$
V_{\alpha}(x) \equiv  V_{(\delta-\alpha)/2,\alpha}(x) = -2^{\alpha} \frac{\Gamma\left(  \frac{\delta+\alpha}{2}\right)}
{\Gamma(\tfrac{\delta-\alpha}{2})}{\left(1+|x|^2\right)^{-\alpha}},
$$
i.e., a zero eigenvalue exists due to \eqref{exist}-\eqref{exist3}.

\medskip
\noindent
(3) It is sufficient to prove the first half of the statement. As proved in Lemma \ref{lem2_atsuhide},
the commutator $[H_\alpha,A]^0$ is well-defined and $H_\alpha$ belongs to class $C^1(A)$. Take $\varepsilon=
(a-C_\alpha)\alpha/(a+C_\alpha)>0$. We compute, for $f\in C_0^\infty(a,b)$ with $a<b<\infty$,
\begin{align*}
& f(H_\alpha)i[H_\alpha,A]^0 f(H_\alpha)=f(H_\alpha)
\left(\alpha\Psi(p^2/2)-2\alpha C_\alpha|x|^2\langle x\rangle^{-2\alpha-2}\right)f(H_\alpha)\\
& \quad=
\varepsilon f(H_\alpha)\Psi(p^2/2)f(H_\alpha)\\
&\qquad
+f(H_\alpha)\left((\alpha-\varepsilon)H_\alpha+(\alpha-\varepsilon)C_\alpha\langle x\rangle^{-2\alpha}-2\alpha C_\alpha|x|^2
\langle x\rangle^{-2\alpha-2}\right)f(H_\alpha).
\end{align*}
By the estimates
\begin{align*}
& a(\alpha-\varepsilon)+(\alpha-\varepsilon)C_\alpha\langle x\rangle^{-2\alpha}-2\alpha
C_\alpha|x|^2\langle x\rangle^{-2\alpha-2}\\
&\quad=
\langle x\rangle^{-2\alpha-2}\left(a(\alpha-\varepsilon)\langle x\rangle^{2\alpha+2}-(\alpha+\varepsilon)
C_\alpha\langle x\rangle^2+2\alpha C_\alpha\right)>2\alpha C_\alpha\langle x\rangle^{-2\alpha-2}>0,
\end{align*}
noting that $a(\alpha-\varepsilon)=(\alpha+\varepsilon)C_\alpha$ and $(\alpha+\varepsilon)C_\alpha(\langle x\rangle^{2\alpha+2}
-\langle x\rangle^2)>0$, we obtain
\begin{equation*}
f(H_\alpha)i[H_\alpha,A]^0 f(H_\alpha)>\varepsilon f(H_\alpha)\Psi(p^2/2)f(H_\alpha)>a\varepsilon f(H_\alpha)^2.
\end{equation*}
The strict Mourre estimate (see \eqref{strict_mourre}) implies the absence of eigenvalues in $(a,b)$.

\medskip
\noindent
(4) We show that for sufficiently large $C_\alpha$ the discrete spectrum is non-empty. Recall that the Dirichlet
form $(\mathcal E, \Dom(\mathcal E))$ of $(-\Delta)^{\alpha/2}$ is given by
$$
\mathcal E(f,g) = \int_{\R^d} |y|^\alpha \widehat{f}(y) \overline{\widehat{g}(y)} dy, \quad f, g \in \Dom(\mathcal E),
$$
where $\Dom(\mathcal E)=\big\{f \in L^2(\R^d): \int_{\R^d} |y|^\alpha |\widehat{f} (y)|^2 dy < \infty \big\}$. Since
the potential $V_\alpha$ is decaying at infinity, it is known that $\Spec(H_\alpha) = \Spec_{\rm ess}(H_\alpha) \cup
\Spec_{\rm d}(H_\alpha)$, where $\Spec_{\rm ess}(H_\alpha) = \Spec_{\rm ess}((-\Delta)^{\alpha/2}) = [0 ,\infty)$,
$\Spec_{\rm d}(H_\alpha) \subset (-\infty, 0)$, and $\Spec_{\rm d}(H_\alpha)$ consists of isolated eigenvalues of finite
multiplicity whenever it is non-empty. Also, by general arguments, whenever there exists $f \in \Dom(\cE)$ such that
$\mathcal E(f,f) + \int_{\R^d} V(x)f^2(x)dx <0$, it follows that $\Spec_{\rm d}(H_\alpha) \neq \emptyset$. To choose
a suitable test function, let $f_r$ and $\lambda_r$ be the principal Dirichlet eigenfunction and eigenvalue,
respectively, for the fractional Laplacian for a ball $B_r(0)$ centered in the origin, and denote the corresponding
quadratic form  by $(\cE_0, \Dom(\cE_0))$. Then we have $\supp f_r = B_r(0)$, $\left\|f_r\right\|_2=1$, and $f_r$ is a
non-increasing radial function by \cite[Def. 1.3, Cor. 2.3]{AL89}. Since $\cE(f_r,f_r) \leq c \cE_0(f_r,f_r) = c
\lambda_r$, with a constant $c>0$, we get
$$
\cE(f_r,f_r) - C_\alpha \int_{\R^d} \frac{f_r^2(x)}{(1+|x|^2)^\alpha}dx
\leq
c \lambda_r - C_\alpha \int_{B_r(0)} \frac{f_r^2(x)}{(1+|x|^2)^\alpha}dx
\leq
c \lambda_r - \frac{C_\alpha}{(1+r^2)^\alpha}.
$$
Thus by choosing $C_\alpha > B(\alpha) = c \inf_{r\geq 0}\lambda_r (1+r^2)^\alpha$, we have at least one negative
eigenvalue.
\end{proof}

\begin{remark}
\label{rem3_atsuhide}
{\rm
We note that the number of negative eigenvalues for the potential in Theorem \ref{dicho} is finite,
in particular, zero is not an accumulation point. Indeed, by the Lieb-Thirring inequality \cite{HL}
we have for the number $N_0$ of negative eigenvalues
\begin{equation*}
N_0 \leq L_\alpha \int_{\R^d} |V_\alpha(x)|^{3/\alpha}dx = 4\pi L_\alpha\ome_dC_\alpha^{3/\alpha}I_3,
\end{equation*}
where $L_\alpha>0$ and $I_3$ denotes $I_\alpha$ with $\alpha$ extended to $3$.}
\end{remark}

\begin{example}\label{ex6_atsuhide}
{\rm
Let $\frac{3}{2} <\alpha<2$ for $d=3$, and
\begin{equation*}
V(x)=-\frac{C}{(1+|x|^2)^{\beta/2}}
\end{equation*}
with $C>0$ and $0<\beta\leq\alpha$. If $C<A(\beta/2)$ with the same constant as in (1) of Theorem \ref{dicho},
then $H=(-\Delta)^{\alpha/2}+V$ has no eigenvalue in virtue of Examples \ref{EX3_edited_by_atsuhide}, \ref{ex4_atsuhide} and (1) of Theorem \ref{dicho}.}
\end{example}

\bigskip
\noindent
\textbf{Acknowledgments}
AI thanks JSPS KAKENHI (Grant Numbers JP20K03625, JP21K03279 and JP21KK0245) and Tokyo University of Science Grant for
International Joint Research for support. JL gratefully thanks the hospitality of IHES, Bures-sur-Yvette, and both thank
the Isaac Newton Institute for Mathematical Sciences, Cambridge, for support (EPSRC grant no EP/K032208/1) and hospitality
during the programme ``Fractional Differential Equations"  where work on this paper was undertaken. IS thanks JSPS KAKENHI
(Grant Numbers JP16 K17612 and JP20K03628) for support. The authors thank Professor Masahito Ohta of
Tokyo University of Science for his valuable comments, and also thank reviewers.

\end{document}